\documentclass[11pt,reqno]{amsart}
\RequirePackage[OT1]{fontenc}
\RequirePackage{amsthm,amsmath}
\RequirePackage[numbers]{natbib}
\RequirePackage[colorlinks,linkcolor=blue,citecolor=blue,urlcolor=blue]{hyperref}
\usepackage{amssymb}
\usepackage{anysize}
\usepackage{hyperref}
\usepackage{extarrows}
\usepackage{multirow}
\usepackage{natbib}
\usepackage{stmaryrd} 
\usepackage{color}
\usepackage{soul}
\usepackage{todonotes}
\usepackage[perpage]{footmisc}
\usepackage{adjustbox}
\usepackage{cancel} 
\usetikzlibrary{matrix,decorations.pathreplacing} 
\usepackage{enumerate} 
\usepackage{amsaddr} 

\numberwithin{equation}{section}
\newtheorem{thm}{Theorem}[section]
\newtheorem{lem}[thm]{Lemma}

\newtheorem{cor}[thm]{Corollary}
\newtheorem{defi}[thm]{Definition}
\newtheorem{rem}[thm]{Remark}

\newcommand{\be}{\begin{equation}}
\newcommand{\ee}{\end{equation}}
\newcommand{\bea}{\begin{eqnarray*}}
\newcommand{\eea}{\end{eqnarray*}}

\newcommand{\eqby}[1]{\mathrel{\stackrel{#1}{=}}}
\newcommand{\eqbydef}{\mathrel{\stackrel{\hspace*{-3mm}\text{\tiny(def)}\hspace*{-3mm}}{=}}}
\newcommand{\leqby}[1]{\mathrel{\stackrel{#1}{\leq}}}

\newcommand{\deq}{\mathrel{\mathop:}=}
\newcommand\numberthis{\addtocounter{equation}{1}\tag{\theequation}} 

\newcommand{\ofN}{}

\makeatletter
\newcommand{\biggg}{\bBigg@\thr@@}
\newcommand{\Biggg}{\bBigg@{3.5}}

\makeatother

\makeatletter

\newcommand{\Rmnum}[1]{\expandafter\@slowromancap\romannumeral #1@}
\makeatother
\makeatletter 
\@addtoreset{equation}{section}
\makeatother  
\renewcommand\theequation{{\thesection}%
                   .{\arabic{equation}}}

\newcommand{\wt}{\widetilde}

\newcommand{\beqa}{\begin{eqnarray}}
\newcommand{\eeqa}{\end{eqnarray}}
\newcommand{\eps}{\varepsilon}

\newcommand{\rd}{{\rm d}}

\renewcommand{\Im}{\mathsf{Im}\,}
\newcommand{\supp}{\mathsf{supp}\,}

\newcommand{\im}{{\mathfrak{Im} \, }}
\newcommand{\E}{{\mathbb E }}
\newcommand{\R}{{\mathbb R }}
\newcommand{\N}{{\mathbb N}}

\renewcommand{\P}{{\mathbb P}}

\newcommand{\C}{{\mathbb C}}

\newcommand{\e}{\varepsilon}

\marginsize{30mm}{30mm}{30mm}{28mm}

\begin{document}
\title{Bounds on the Norm of Wigner-type Random Matrices}
\author{L\'aszl\'o Erd{\H{o}}s \and Peter M{\"u}hlbacher}
\thanks{Partially supported by ERC Advanced Grant RANMAT No. 338804}
\address{\textsc{IST Austria, A-3400 Klosterneuburg}}\author{\vspace{-1.2cm}}
\email{\url{lerdos@ist.ac.at}}
\email{\url{peter@muehlbacher.me}}
\date{\today}
\maketitle

\begin{abstract}
	We consider a Wigner-type ensemble, i.e. large hermitian  $N\times N$ random matrices $H=H^*$ with centered independent entries and with a general matrix of variances  $S_{xy}=\E|H_{xy}|^2$. The norm of $H$ is asymptotically  given by the maximum of the support of the self-consistent density of states. We establish a bound on this maximum in terms of norms of powers of $S$ that substantially improves the earlier bound $2\| S\|^{1/2}_\infty$ given in \cite{ELK16}. The key element of the proof is an effective Markov chain approximation for the contributions of the weighted Dyck paths appearing in the iterative solution of the corresponding Dyson equation.
\end{abstract}

\section{Introduction and the main result}

Large hermitian random matrices with independent entries tend to exhibit deterministic patterns.
In particular, the empirical density of eigenvalues typically converges to a deterministic density profile, $\rho$, called the {\it self-consistent density of states}, that can be determined by solving a system of quadratic equations. Under very general conditions, $\rho$ is compactly supported and the largest eigenvalue  of the random matrix is  asymptotically given by the maximum  of the support of $\rho$. 

In the simplest case of $N\times N$ Wigner matrices, i.e. when $H$ has centered, identically distributed entries that are independent  (up to the symmetry constraint $H=H^*$), the self-consistent density of states is given explicitly by the Wigner semicircle law. Under the customary normalization $\E |H_{xy}|^2= N^{-1}$, the semicircle distribution is supported in $[-2,2]$. With very high probability (and also almost surely) the Euclidean matrix norm tends to 2, i.e.  $\| H\|_2\to 2$ as $N\to \infty$, assuming the fourth moment of $\sqrt{N}H_{xy}$ is finite \cite{BY88}.

In this paper we consider {\it Wigner type} matrices introduced in  \cite{ELK16}. These are generalizations of the Wigner ensemble where independence of the matrix elements is retained but their distribution may vary within the matrix. We assume  $\E H=0$.
The relevant parameter of the model is the matrix of variances 
$$
  S=(S_{xy})_{x,y=1}^N, \qquad S_{xy}:  =\E |H_{xy}|^2.
$$
The self-consistent density of states is obtained via the  solution of a system of quadratic equations
\be\label{Qve}
      -\frac{1}{m_x} = z + \sum_{y=1}^N S_{xy} m_y\quad\quad\quad x=1,2,\dots,N,
\ee
where $z\in {\mathbb H}$ is a complex spectral parameter in the upper half plane. This equation was extensively studied in \cite{QVE, QVEMem}.
Under the additional condition that $\im m_x> 0$, the solution to \eqref{Qve} is unique and it depends analytically on $z\in  {\mathbb H}$. Its average, 
\be\label{def:m}
 m(z): =\frac{1}{N}\sum_{x=1}^N m_x (z)
\ee
is the Stieltjes transform of a probability density $\rho$. This relation defines the self-consistent density of states measure that can be obtained by inverting the Stieltjes transform as
\be\label{def:rhoN} \rho (\rd\tau)\deq\lim_{\eta\downarrow 0}\frac{1}{\pi N}\sum_{x=1}^N\Im m_x(\tau+i\eta)\rd\tau.
\ee
A simple symmetry argument shows that $m(-\bar z) = -\overline{m(z)}$ and thus $\rho$ is an even measure.  We remark that under  additional assumptions on $S$, the measure $\rho$ is absolutely continuous with a H\"older continuous density except at $\tau=0$, where it may have a Dirac delta component (Corollary 7.4 \cite{QVEMem}). 
Note that $m_x,  m, \rho$ as well as $S$ depend on $N$,  i.e. $m_x=m_x^{(N)}$ etc., but this dependence will sometimes be omitted from the notation.

Under very general conditions on $S$ and some higher moment assumption of $H$, it is well known that the empirical density of eigenvalues of $H$ is asymptotically given by $\rho$. This holds not only on the global scale, but even on very small  scales  slightly above the typical eigenvalue spacing; these statements are called {\it local laws} for Wigner-type matrices (Theorem 1.7 of \cite{ELK16}). 

Local laws  are typically not  sensitive to individual eigenvalues except at the spectral edges, where a stronger version of the local law holds. Therefore,  the maximum of the support of $\rho$ correctly describes the largest eigenvalue  or the norm of $H$ (see \cite{AEKN} for a quite general setup). 
In particular, the norm of a Wigner matrix with the above normalization converges to 2. The speed of convergence has been addressed in several papers in increasing generality, see e.g. \cite{BY88,Vu05,BS12,EKY13}. Similarly, the norm of a Wigner-type matrix converges to $\max \supp \rho$. For a general variance matrix $S$, neither $\rho$ nor its support can be computed explicitly; our current goal is to give a good bound on $\max\supp\rho$.
A relatively simple argument  (see Proposition 2.1 \cite{ELK16}) gives 
\be\label{trivbound}
\max \supp \rho\le 2 \| S\|^{1/2},
\ee
 where for any matrix $M$ we let $\|M\|:=\| M\|_\infty:= \max_x \sum_y |M_{xy}|$ denote the matrix norm induced by the maximum norm on $\C^N$. In this paper $\|\cdot\|$ always denotes this  maximum norm.

Our  main theorem  considerably improves the bound \eqref{trivbound} and it is still expressed in terms of norms of powers of $S$. 

\begin{thm}\label{thm:main} Let $S$ be a variance matrix, i.e. a symmetric $N\times N$ matrix with nonnegative entries. Let $\rho$ be the self-consistent density of states obtained from the unique solution of \eqref{Qve} via inverse Stieltjes transform \eqref{def:rhoN}. Set 
\be\label{eq:def_zj}
z_j\deq \frac{\|S^j\|}{\|S\|^j}
\ee
for any $j\in \N$. Then for any fixed $J\in\N\cup \{ \infty\}$ we have
\begin{equation}\label{newbound}
	\max\supp\rho\leq 2\frac{\|S\|^\frac{1}{2}}{w_c(J)},
\end{equation}
where $w_c(J)$ is the smallest  positive root of 
the function
\be\label{delphi}
\phi_J(w):=1-\frac{w}{2}\Big(1+\sum_{j=1}^J\left(\frac{w}{2}\right)^jz_j+\sum_{j>J}\left(\frac{w}{2}\right)^j\Big).
\ee
\end{thm}

It is easy to see that $w_c(J)$ is an increasing function of $J$, so a choice of larger $J$ yields a better bound.
In particular $J=\infty$ is the best.
On the other hand, larger $J$  is more computation intensive as it requires to  
compute norms of higher powers of $S$.

Comparing  \eqref{newbound} with \eqref{trivbound}, notice that 
the main source of the improvement is the simple fact that 
the inequality  $\| S^j\|\le \| S\|^j$ rarely saturates. 
Indeed, it is easy to see  that
$w_c$  is a strictly monotonically decreasing function of  all  $z_j$. If
all $\| S^j\|$ norms were
replaced with $\|S\|^j$, i.e. we set $z_j=1$, then $w_c(J)=1$ for any $J$
and the two bounds were identical. Once $z_j<1$ for some $j$,
we have $w_c>1$.

In the Appendix we illustrate in an example the effect of the improvement and compare it with the exact value of $\max \supp \rho$ obtained numerically.


Combining Theorem~\ref{thm:main} with Corollary 2.3 of \cite{EKS} (or Theorem 4.7 of \cite{AEKN})
on the convergence of the 
largest eigenvalue of the Wigner type matrix and using that $w_c(J)$ 
depends continuously and monotonically on $z_j$, we immediately obtain the following

\begin{cor} Let $H=H^{(N)}$ be a sequence of  hermitian $N\times N$ Wigner type matrices, with  centered entries and
matrix of variances $S_{xy} =S_{xy}^{(N)}= \E |H_{xy}|^2$.  Assume that  $S_{xy}\le C^*/N$ for some constant $C^*$, independent of $N$.
Further, we assume a finite moment condition on the matrix elements,
 i.e. that for any $q\in \N$ there is a constant $C_q$, independent of $N$, such that
$$
      \max_{x,y,N}\E \big( \sqrt{N} |H_{xy}| \big)^q \le C_q.
$$
Set 
$$
 z_j:= \limsup_{N\to\infty}  \frac{\|[S^{(N)}]^j\|}{\|S^{(N)}\|^j}
$$
and for any $J\in \N \cup\{\infty\}$ let $w_c(J)$ be the smallest positive root of $\phi_J$ defined in \eqref{delphi}.
Then 
for any $\epsilon>0$ (small) and any $D>0$ (large) we have the following bound on  the largest eigenvalue of $H$:
$$
  \P \Big(  |\lambda_{max} (H^{(N)})|\ge 2\frac{\|S^{(N)}\|^\frac{1}{2}}{w_c(J)} +\epsilon\Big)  \le C(\epsilon, D) N^{-D}
 $$
 for some constant $C(\epsilon, D)$ depending only on $C^*$ and the sequence of constants $C_q$, in addition to $\epsilon$ and $D$.
\end{cor}

Instead of Wigner type matrices, one may also consider Gram matrices, i.e., matrices of the form $H= XX^*$ where
$X$ is an $M\times N$ matrix with  centered indepedent entries (without any symmetry condition) and 
 $S_{ij}:= \E |x_{ij}|^2$  is the matrix of variances. The spectral radius of $H$ is the
 square of the spectral radius of the linearized matrix
 $$
   \mathcal{X}:= \begin{pmatrix} 0 & X \cr X^* & 0 \end{pmatrix}  \quad \mbox{with variance matrix} \quad  {\mathcal S}: = \begin{pmatrix} 0 & S \cr S^t & 0 \end{pmatrix}.
$$
Since $ \mathcal{X}$ is a Wigner type matrix, Theorem~\ref{thm:main} and its corollary directly applies. The
norm $\|\mathcal{S}^j\|$ can be trivially expressed in terms of the norms of matrices of the form $SS^*SS^*...$
and $S^*SS^*S....$.

We remark that very similar questions were studied  independently in a recent  work of M. Ottolini \cite{Otto} who derived a
variational formula for $\max \supp \rho$ and proved the convergence of the largest
eigenvalue to it. This formula is exact, but not explicit in terms of $S$
as it still requires to solve a variational problem. It is an open question to establish  connections
between the  two approaches, especially find an explicit formula, if possible, in terms of $S$  for the solution 
of Ottolini's variational problem.

\bigskip
  
We now explain the main novelty of our approach. We introduce a tree-graph expansion for representing the solution to \eqref{Qve}. 
Unlike in the traditional  proof of the Wigner semicircle law via the moment method, in our case the graphs do not contribute equally; they are weighted by factors of $S_{xy}$ assigned to edges. This defines an $S$-dependent measure ${\mathcal P}_S$ on the space of trees.

We  then estimate the contribution of each tree by chopping  it up into possibly long linear segments. Along the linear pieces, we can perform the summation $\sum_{xyz...uv} S_{xy}S_{yz} ... S_{uv} = (S^j)_{xv}$ explicitly.
This enables us to use the stronger bound $\| S^j\|$ instead of the trivial one $\| S\|^j$.
We present an algorithm for a good chopping. We then compute the expected value of the  corresponding contributions with respect to the measure ${\mathcal P}_S$. It turns out that the relevant regime is the limit as the size of the trees goes to infinite. In this limit we approximate the measure  ${\mathcal P}_S$ by a Markov chain  for the purpose of computing the weighted contributions of all graphs. The approximate Markov structure becomes apparent as we identify the tree graphs with Dyck paths. Finally, in the Markov model we can compute the answer explicitly.

\section{Trees and Dyck paths}

We start with a simple observation that allows us to express $\max\supp\rho$ in terms of the radius of convergence of the Laurent series expansion of the Stieltjes transform of $\rho$.

	Let $\rho$ be a compactly supported, symmetric
	probability measure on the real line with  $\supp\rho\subset [-r, r]$ for  $r :=\max\supp \rho >0$.
	Clearly its Stieltjes transform
 $$m(z)\deq\int_\R\frac{\rho(\rd \tau)}{\tau -z}$$  is analytic on $\C\setminus \supp\rho$. Its Laurent series is written
 as
 $$
      m(z) = -\frac{1}{z} \sum_{k=0}^\infty \Big(\frac{1}{z}\Big)^k \mu_k, \qquad \mu_k:=\int_\R \tau^k \rho(\rd \tau).
 $$
By the Cauchy-Hadamard theorem on the radius of convergence of this power series we immediately  obtain 
\be\label{maxbound}
     \max\supp \rho = \limsup_{k\to\infty}  \mu_k^{1/k}.
\ee
      
We can apply the relation \eqref{maxbound} not only for 
 $m(z)$ defined in \eqref{def:m}, but also for each $m_x(z)$ since it is 
  the Stieltjes transform of some probability measure $\rho_x$ (see, e.g. Theorem 2.1 \cite{QVEMem}).
 Clearly $\rho_x$ is also symmetric and $\rho =\frac{1}{N}\sum_x \rho_x$, thus 
 \be\label{sups}
 \supp\rho = \bigcup_{x=1}^N \supp\rho_x.
 \ee
  In particular, 
 if $|z|> \max\supp\rho$, then not only $m(z)$ has a convergent Laurent series, but each $m_x$ as well:
 \be\label{mx}
 m_x(z) = -\frac{1}{z} \sum_{k=0}^\infty \Big(\frac{1}{z}\Big)^k \mu_{x,k}, \qquad \mu_{x,k}:=\int_\R \tau^k \rho_x(\rd \tau).
 \ee
Similarly, we have
\be\label{maxibound}
     \max\supp \rho = \max_x\max\supp\rho_x = \max_x \limsup_{k\to\infty}  \mu_{x,k}^{1/k}.
\ee
To estimate $\limsup_{k\to\infty} \mu_{x,k}^{1/k}$, we will express $\mu_{x,k}$ in terms of sums of products of the matrix elements of $S$. Similarly to the standard proof of the Wigner semicircle law by the moment method (e.g. Section 2.1 \cite{AGZ}), we represent these sums diagrammatically, via an expansion in terms of Dyck paths.
Since the paths are weighted by $S$, the estimate is not a simple combinatorial enumeration of the Dyck paths. We will see that these weights substantially distort the uniform counting measure on the set of Dyck paths. In the next sections we develop a formalism to bookkeep  and effectively estimate these weights.
In what follows we will use the notations
$$[a,b] = \{i\in\N : a\leq i\leq b\},
	\quad\quad
  [a,b) = \{i\in\N : a\leq i\leq b-1\}.$$

\subsection{Dyck Path}\label{subsec:dyckpaths}

We start with recalling the definition of the Dyck paths:

\begin{defi} Dyck paths of length $2k$ are paths $\pi: [0,2k] \to \N$ such that
$\pi(0)=\pi(2k)=0$, $ |\pi(i)-\pi(i+1)|=1$. Denote the set of Dyck paths of length $2k$ by $D_{2k}$. We say that the $i$-th step is an {\it up-run} if $\pi(i)<\pi(i+1)$ and a {\it down-run} otherwise. 
\end{defi}

Alternatively, Dyck paths encode algebraically legitimate bracketing of a product of $2k$ non-associative symbols in a straightforward manner.
It is sufficient to bookkeep the brackets only. Thus we consider a string consisting of $k$ opening and $k$ closing brackets in such a way that for every {\it up-run} we append a "(" to the string, for every {\it down-run} append a ")".

We now recall the {\it tree representation} of the Dyck paths.				
Let $\mathcal T_k$ denote the set of planar, rooted, undirected trees $\Gamma=\big(V(\Gamma),E(\Gamma)\big)$ with $|E(\Gamma)|=k$.

We always draw a tree in the plane in such a way that the vertices at the same distance from the root are drawn at the same height (horizontal level) relative to the root and the root is the lowest point (Figure \ref{fig:labelledTreeAndDyckPathRepresentation}). The height function is denoted by $h(v)$ for all $v\in V(\Gamma)$, $\Gamma=(V(\Gamma),E(\Gamma))\in\mathcal T_k$; we set $h(\text{root})=0$. 
In this way every vertex $v$ (apart from the root) has a unique {\it father}, i.e. an adjacent vertex of height $h(v)-1$, and may have some {\it children}; these are adjacent vertices with height $h(v)+1$, whose number we denote by $c(v)$. Unless $v$ is a root, we have $c(v)=d(v)-1$, where $d(v)$ is the {\it degree} of $v$, i.e. the number of adjacent vertices. Vertices with no children are called {\it leaves}.

Every edge $e\in E(\Gamma)$ has two vertices adjacent to it, denoted by $e_-, e_+ \in V(\Gamma)$, the sign indicating their relative  height, i.e. $h(e_-)< h(e_+)$. We also extend the height function to edges by setting $h(e)\deq h(e_+)$.

The planarity imposes an orientation on every tree. In particular, it is possible to walk around the outer boundary of $\Gamma$ (say, in clockwise direction) starting and arriving at the root. 
In this way, for any element $\Gamma$ of $\mathcal T_k$ we can assign an element $\pi(\Gamma)$ of $D_k$ in such a way that we set $\pi(i)$ to be the distance to the root at the $i$-th step of this walk. This map is clearly a bijection for each fixed $k$. We define 
$\Gamma:\bigcup_{k=0}^\infty D_{2k}\to\bigcup_{k=0}^\infty \mathcal T_k$ to be  the inverse of this map. For any Dyck path $\pi$, we call $\Gamma(\pi)$ ``the tree corresponding to $\pi$" (see Figure \ref{fig:labelledTreeAndDyckPathRepresentation} for an example).
					

A finite collection of several disjoint trees is called {\it forest}. The set of forests with a total of $k$ edges is denoted by $\mathcal F_k\deq\big\{\{\Gamma_i\}_{i} : \Gamma_i\in\mathcal T_{m_i} \text{ with }\sum_i m_i=k\big\}$.
Every component $\Gamma_i$ has a single root which is drawn as its lowest vertex.
For any forest $\Gamma$, the set of roots is denoted by $R(\Gamma)\subseteq V(\Gamma)$.
For general forests not all roots will be drawn at the same horizontal level.
The vertices are drawn as ``bullets" with the convention that roots are unfilled $(\circ)$ and filled $(\bullet)$ otherwise.


\begin{figure}[h]
	\begin{tikzpicture}[baseline=8ex]
		\node at (-0.3,0) {$a$};
		\node at (-0.3,1) {$b$};
		\node at (-0.3,2) {$c$};
		\node at (-0.7,3) {$d$};
		\node at (0.7,3) {$e$};
		\node at (0.7,2) {$f$};
		\draw[thick] (0,0) -- (0,1) -- (0,2) -- (-1,3);
		\draw[thick] (0,2) -- (1,3);
		\draw[thick] (0,1) -- (1,2);
		\draw[fill=white] (0,0) circle (3pt);
		\draw[fill=black] (0,1) circle (3pt);
		\draw[fill=black] (0,2) circle (3pt);
		\draw[fill=black] (-1,3) circle (3pt);
		\draw[fill=black] (1,3) circle (3pt);
		\draw[fill=black] (1,2) circle (3pt);
	\end{tikzpicture}=	
	\begin{tikzpicture}[baseline=8ex]
		\draw[fill=black] (0,0) circle (3pt);
		\draw[fill=black] (1,1) circle (3pt);
		\draw[fill=black] (2,2) circle (3pt);
		\draw[fill=black] (3,3) circle (3pt);
		\draw[fill=black] (4,2) circle (3pt);
		\draw[fill=black] (5,3) circle (3pt);
		\draw[fill=black] (6,2) circle (3pt);
		\draw[fill=black] (7,1) circle (3pt);
		\draw[fill=black] (8,2) circle (3pt);
		\draw[fill=black] (9,1) circle (3pt);
		\draw[fill=black] (10,0) circle (3pt);
		\node at (0.3,0) {$a$};
		\node at (1.3,1) {$b$};
		\node at (2.3,2) {$c$};
		\node at (3.3,3) {$d$};						
		\node at (4.3,2) {$c$};
		\node at (5.3,3) {$e$};						
		\node at (6.3,2) {$c$};
		\node at (7.3,1) {$b$};
		\node at (8.3,2) {$f$};
		\node at (9.3,1) {$b$};
		\node at (9.7,0) {$a$};
		\draw[thick] (0,0) -- (1,1) -- (2,2) -- (3,3) -- (4,2) -- (5,3) -- (6,2) -- (7,1) -- (8,2) -- (9,1) -- (10,0);
	\end{tikzpicture}
	\caption{Tree and Dyck path representation of $\Gamma= \Gamma\big(\text{``((()())())"}\big)\in\mathcal T_5$.}\label{fig:labelledTreeAndDyckPathRepresentation}
\end{figure}
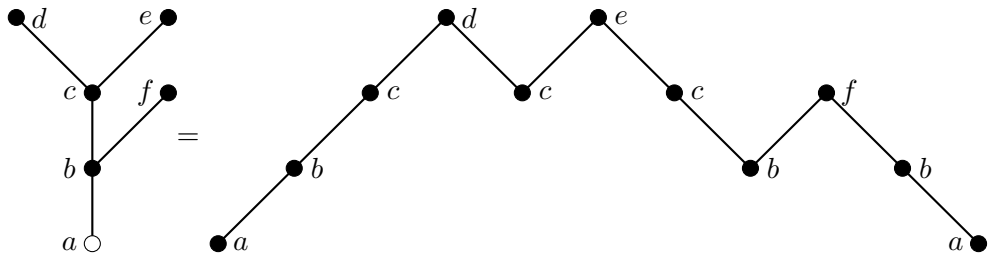

\subsection{Graphical Representation of the Dyck Path Expansion}  

We now introduce a graphical representation to rewrite $\mu_{x,k}$. The same expansion in a slightly different presentation was also used in \cite{Otto}.


\begin{lem}[Graphical representation of $m_x\ofN(z)$]\label{lem:graphicalRepresentationOfm_j}
For every $\Gamma=(V(\Gamma),E(\Gamma))\in\mathcal T_k$, and for every $x\in\{1,\dots,N\}$ set \be\label{def:valNj} val\ofN_x(\Gamma)\deq \Big(\prod_{v\in V(\Gamma)}\sum_{x_v=1}^N\Big)\Big[\delta_{x_\text{root}=x} \prod_{e\in E(\Gamma)}S_{x_{e_-}x_{e_+}}\Big].\ee 
Then for the Laurent series \eqref{mx} we have
	\be
	  m_x\ofN(z)=-\frac{1}{z}\sum_{k=0}^\infty \Big[\sum_{\Gamma\in\mathcal T_k}val\ofN_x(\Gamma)\Big]z^{-2k}, \qquad
	  |z|>\max\supp\rho_x.
	\ee
\end{lem}

\begin{proof}
Introduce $u_x(z)\deq -zm_x(z)$ and note that
the QVE \eqref{Qve} is equivalent to
\be\label{eq:combinatorialuxRecursion}
	u_x = 1+z^{-2}\sum_{y=1}^N S_{xy}u_xu_y.
\ee
Using \eqref{mx}, $u_x(z)$ admits a Laurent series expansion for $|z|$ large enough. By the symmetry of the measure $\rho_x$, the odd coefficients vanish and with $c_{x,k}\deq\mu_{x,2k}$ we have 
\be\label{eq:uxAnsatz}
	u_x(z) = \sum_{k=0}^\infty c_{x,k}z^{-2k},
\ee
for large $|z|$.
Now plugging \eqref{eq:uxAnsatz} into \eqref{eq:combinatorialuxRecursion} and comparing coefficients
we get the following recursion: 
\be\label{eq:inductionRecursion}
c_{x,k} = \sum_{y=1}^N S_{xy}\sum_{n=0}^{k-1}c_{x,k-n-1}c_{y,n}.
\ee
To show that $c_{x,k}=\sum_{\Gamma\in\mathcal T_k}val_x(\Gamma)$ we proceed by induction on $k$. The base case is clear. Assume that for all $n$ with $0\leq n<k$ we already know that
\be\label{eq:inductionHypothesis}
c_{x,n} = \sum_{\Gamma\in\mathcal T_n}val_x(\Gamma)\quad\quad\text{for } x=1,\dots,N.
\ee
We identify every $\Gamma\in\mathcal T_k$ with its Dyck path $\pi(\Gamma)\in D_{2k}$. 
We define for every $\Gamma\in\mathcal T_k$ the numbers $n_1=n_1(\Gamma), n_2=n_2(\Gamma)$ by 
$$2n_1\deq\max\{t\in [0,2k):\pi(\Gamma)(t)=0\},$$
and $n_2\deq k-n_1-1$.
Now every $\Gamma\in\mathcal T_k$ can \emph{uniquely} be written as $\Gamma = \Gamma_1\oplus\Gamma_2$ with $\Gamma_1\in \mathcal T_{n_1},\Gamma_2\in\mathcal T_{n_2}$, where we define $\Gamma_1\oplus\Gamma_2$ via its Dyck path representation as follows:
$$\pi(\Gamma_1\oplus\Gamma_2)(t)\deq \begin{cases}
		\pi(\Gamma_1)(t), & \text{if } t=0,1,\dots,2|E(\Gamma_1)| \\
		\pi(\Gamma_2)(t)+1, & \text{if } t=2|E(\Gamma_1)|+1,\dots,2|E(\Gamma_1)|+2|E(\Gamma_2)|+1\\
		0, & \text{if } t=2|E(\Gamma_1)|+2|E(\Gamma_2)|+2.
	\end{cases}$$
See Figure \ref{fig:gammaOplusGammaIllustration} for an illustration.
By definition \eqref{def:valNj}, we have 
\be\label{eq:recursionOfVal}
\sum_{y=1}^N S_{xy}val_x(\Gamma_1)val_y(\Gamma_2) = val_x(\Gamma_1\oplus\Gamma_2).
\ee
Plugging in \eqref{eq:inductionHypothesis} into \eqref{eq:inductionRecursion} and using \eqref{eq:recursionOfVal} we see that
\be\label{eq:sumOverConcatTrees}
	c_{x,k}=\sum_{y=1}^N S_{xy}\sum_{n=0}^{k-1} val_x(\Gamma_1)val_y(\Gamma_2)
	\eqby{\eqref{eq:recursionOfVal}} \sum_{n=0}^{k-1}\sum_{\Gamma_1\in\mathcal T_{n}}\sum_{\Gamma_2\in\mathcal T_{k-n-1}}val_x(\Gamma_1\oplus\Gamma_2).
\ee
Since for every $\Gamma\in\mathcal T_k$ there exists exactly one pair $n_1(\Gamma),n_2(\Gamma)$ such that $\Gamma=\Gamma_1\oplus\Gamma_2$ with uniquely determined $\Gamma_i\in\mathcal T_{n_i}$, we have that \eqref{eq:sumOverConcatTrees} is just $\sum_{\Gamma\in\mathcal T_k}val_x(\Gamma)$.
\end{proof}

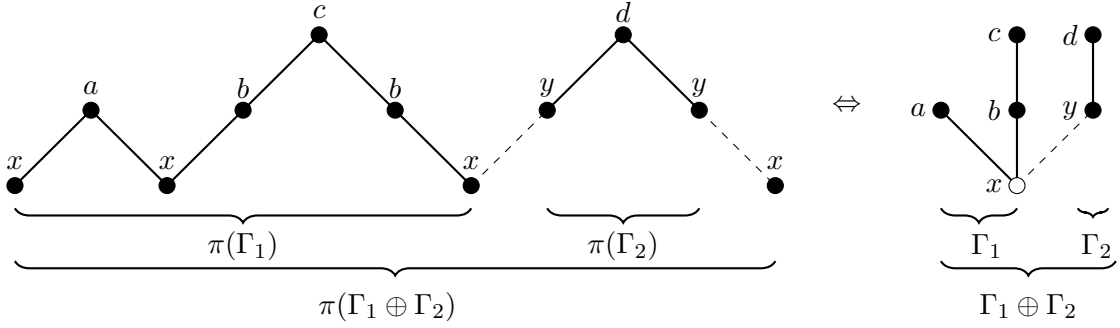
\begin{figure}[h]
\begin{tikzpicture}[baseline=6ex]
	\draw[fill=black] (0,0) circle (3pt);
	\draw[fill=black] (1,1) circle (3pt);
	\draw[fill=black] (2,0) circle (3pt);
	\draw[fill=black] (3,1) circle (3pt);
	\draw[fill=black] (4,2) circle (3pt);
	\draw[fill=black] (5,1) circle (3pt);
	\draw[fill=black] (6,0) circle (3pt);

	\draw[fill=black] (7,1) circle (3pt);
	\draw[fill=black] (8,2) circle (3pt);
	\draw[fill=black] (9,1) circle (3pt);

	\draw[fill=black] (10,0) circle (3pt);

	\draw[thick] (0,0) -- (1,1) -- (2,0) -- (3,1) -- (4,2) -- (5,1) -- (6,0);
	\draw[dashed] (6,0) -- (7,1);
	\draw[thick] (7,1) -- (8,2) -- (9,1);
	\draw[dashed] (9,1) -- (10,0);
	
	\draw [decoration={brace,amplitude=0.5em},decorate, thick]
        (6,-.3) -- (0,-.3);
    \node at (3,-.8) {$\pi(\Gamma_1)$};

    \draw [decoration={brace,amplitude=0.5em},decorate, thick]
        (9,-.3) -- (7,-.3);
    \draw [decoration={brace,amplitude=0.5em},decorate, thick]
        (10,-1) -- (0,-1);
    \node at (4.9,-1.6) {$\pi(\Gamma_1\oplus\Gamma_2)$};

    \node at (8,-.8) {$\pi(\Gamma_2)$};
    
    \node at (0,.3)  {$x$};
    \node at (1,1.3) {$a$};
    \node at (2,.3)  {$x$};
    \node at (3,1.3) {$b$};
    \node at (4,2.3) {$c$};
    \node at (5,1.3) {$b$};
    \node at (6,.3)  {$x$};
    \node at (7,1.3) {$y$};
    \node at (8,2.3) {$d$};
    \node at (9,1.3) {$y$};
    \node at (10,.3) {$x$};
\end{tikzpicture}
$\quad\Leftrightarrow\quad$
\begin{tikzpicture}[baseline=6ex]
	\node at (-0.3,0) {$x$};
	\node at (-1.3,1) {$a$};
	\node at (-0.3,1) {$b$};
	\node at (-0.3,2) {$c$};
	\node at (0.7,1) {$y$};
	\node at (0.7,2) {$d$};
	\draw[thick] (1,2) -- (1,1);
	\draw[thick] (0,0) -- (0,2);
	\draw[dashed] (0,0) -- (1,1);
	\draw[thick] (0,0) -- (-1,1);
	\draw[fill=white] (0,0) circle (3pt);
	\draw[fill=black] (0,1) circle (3pt);
	\draw[fill=black] (0,2) circle (3pt);
	\draw[fill=black] (-1,1) circle (3pt);
	\draw[fill=black] (1,1) circle (3pt);
	\draw[fill=black] (1,2) circle (3pt);
	
	\draw [decoration={brace,amplitude=0.5em},decorate, thick]
        (0,-.3) -- (-1,-.3);
	\draw [decoration={brace,amplitude=0.3em},decorate, thick]
        (1.2,-.3) -- (0.8,-.3);
	\node at (-.4,-.8) {$\Gamma_1$};
	\node at (1.05,-.8) {$\Gamma_2$};
    \draw [decoration={brace,amplitude=0.5em},decorate, thick]
        (1.3,-1) -- (-1,-1);
	\node at (.15,-1.6) {$\Gamma_1\oplus \Gamma_2$};
\end{tikzpicture}
\caption{Illustration of $\pi(\Gamma_1\oplus\Gamma_2)$ and $\Gamma_1\oplus\Gamma_2$.} \label{fig:gammaOplusGammaIllustration}
\end{figure}

Combining \eqref{maxibound} with Lemma \ref{lem:graphicalRepresentationOfm_j} we get
\be\label{eq:suppRhoNbddByLimsupValN}
	\max\supp\rho\ofN\leq \limsup_k\Bigg|\sum_{\Gamma\in\mathcal T_k}val\ofN(\Gamma)\Bigg|^\frac{1}{2k},
\ee
for \be\label{def:valOfTree}val\ofN \deq \max_{x=1,\dots,N}val\ofN_x.\ee



\begin{rem}[Obtaining the already known bound $2\|S\|^\frac{1}{2}$]\label{rem:trivialBound}
Ignoring the internal structure of the trees $\Gamma\in\mathcal T_k$ and successively summing up the labels starting from the leaves of the tree, using $\sum_y S_{xy}\le \| S\|$, we can bound $val\ofN(\Gamma)$ by $\|S\|^k$.
Noting that $|\mathcal T_k|=|D_{2k}|=O(2^{2k})$ as $k\to\infty$ we easily get the claimed bound from \eqref{eq:suppRhoNbddByLimsupValN}.
\end{rem}

\subsection{Intuition for improvement} \label{rem:intuitionHowToImproveTrivialBoundOnValPi}

	
We explain in a simple example how to improve the previous trivial bound. For example, for $\Gamma$ as in Figure \ref{fig:labelledTreeAndDyckPathRepresentation} we have $$val\ofN(\Gamma)=\max_x \sum_{yzuvw} S_{xy}S_{yz}S_{zu} S_{zv}S_{yw}.$$ Instead of simply bounding it by $\|S\|^5$ we could bound it by
\be\label{eq:exampleBndUp}
val\ofN(\Gamma)
\leq \max_x \sum_{yzu} S_{xy}S_{yz}S_{zu}\max_{z'} \sum_v S_{z'v} \max_{y'} \sum_w S_{y'w}
\leq \|S^3\| \|S\|^2
\ee
or by 
\be\label{eq:exampleBndDown}
val\ofN(\Gamma)
\leq\max_x \sum_{yw} S_{xy}S_{yw} \max_{y'} \sum_{zv} S_{y'z}S_{zv} \max_{z'} \sum_u S_{z'u}
\leq \|S^2\|^2 \|S\|,
\ee
both of which are less or equal than $\|S\|^5$ since the norm is submultiplicative.
	
It is easy to see that such a process always gives a bound of the form $val\ofN(\Gamma)\leq\prod_{i\leq k}\|S^i\|^{p_i}$ with some sequence of natural numbers $(p_i)_i$ such that $\sum_{i}ip_i=k$, depending on $\Gamma$. For certain $\Gamma$'s the improvement over the trivial bound is meagre or even non-existent. While we have a certain freedom in ``chopping up" the multiple summation for $val\ofN(\Gamma)$, in general we cannot obtain all bounds of the form $val\ofN(\Gamma)\leq\prod_{i\leq k}\|S^i\|^{p_i}$. 
Our current choice of $\Gamma$, for example, does not admit the bounds $\|S^4\|\|S\|$ or $\|S^5\|$; these would require a path of length $4$ or $5$, respectively, from the root. 
The worst case for $k=5$, the tree corresponding to $\pi = ()()()()()$ (every node is connected to the root), does not admit any bound other than $\|S\|^5$. 
	
In the next sections we first formalise the above process of chopping up $val\ofN$. Then we quantify how trees between the two extreme cases (the completely linear tree with $d(v)\leq 2$ for every vertex $v$ and the tree where $d(v_{root})=k$) typically look like for large $k$ and which chopping up gives the "best" (smallest) weight.

It is also worth noting that for general $S$ we cannot say that one bound is always better than the other one. Depending on $S$ either $\|S^3\|\|S\|^2$ or $\|S^2\|^2\|S\|$ may be preferable.  In what follows we simply choose a fixed method (independent of $S$) to obtain our bound.

\subsection{The Chopping Up Process}

\subsubsection{Introducing the chopping-up operation}
Recall that $\mathcal F_k$ denotes the set of forests (collection of rooted trees drawn in the plane according to the convention of Section \ref{subsec:dyckpaths}) with $k$ edges in total.

Chopping up is an operation $\mathcal F_k\to\mathcal F_k$ for every $k$ where some vertices of $\Gamma\in\mathcal F_k$ are {\it split} but the edge set remains unchanged.
{\it Splitting of a vertex $v$} is an operation that creates a few new copies of $v$ and  disconnects some (or all) edges emanating from $v$ in the upward direction in such a way that these edges will emanate from a new copy of $v$. The new vertices (called {\it copies} of $v$ in the splitting), together with the original $v$ that is kept, are drawn next to each other in an oriented fashion to keep the planarity of the graph, see Figure \ref{fig:choppingUp} for possible splittings of the tree in Figure \ref{fig:labelledTreeAndDyckPathRepresentation}.
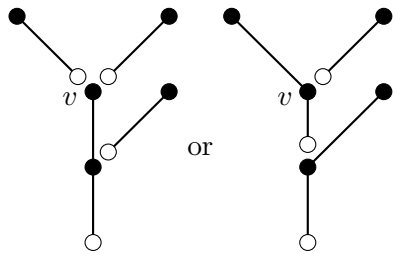
\begin{figure}[h]
	\begin{tikzpicture}[baseline=7ex]
		\draw[thick] (0,0) -- (0,1) -- (0,2);
		\draw[thick] (-0.2,2.2) -- (-1,3);
		\draw[thick] (0.2,2.2) -- (1,3);
		\draw[thick] (0.2,1.2) -- (1,2);		
		\draw[fill=white] (0,0) circle (3pt);
		\draw[fill=black] (0,1) circle (3pt);
		\draw[fill=white] (0.2,1.2) circle (3pt);
		\draw[fill=black] (0,2) circle (3pt);
		\draw[fill=white] (-0.2,2.2) circle (3pt);
		\draw[fill=white] (0.2,2.2) circle (3pt);
		\draw[fill=black] (-1,3) circle (3pt);
		\draw[fill=black] (1,3) circle (3pt);
		\draw[fill=black] (1,2) circle (3pt);
		
		\node at (-.3,1.9){$v$};
	\end{tikzpicture}
	or
	\begin{tikzpicture}[baseline=7ex]
		\draw[thick] (0,0) -- (0,1) -- (1,2);
		\draw[thick] (0,1.2) -- (0,2) -- (-1,3);
		\draw[thick] (0.2,2.2) -- (1,3);
		\draw[fill=white] (0,0) circle (3pt);
		\draw[fill=black] (0,1) circle (3pt);		
		\draw[fill=white] (0,1.3) circle (3pt);
		\draw[fill=black] (0,2) circle (3pt);
		\draw[fill=white] (0.2,2.2) circle (3pt);
		\draw[fill=black] (-1,3) circle (3pt);
		\draw[fill=black] (1,3) circle (3pt);
		\draw[fill=black] (1,2) circle (3pt);

		\node at (-.3,1.9){$v$};
	\end{tikzpicture}
	\caption{Two possible splittings of the tree in Figure \ref{fig:labelledTreeAndDyckPathRepresentation}. The vertex $v$ is completely split in the left graph, and it is split in the leftmost way (almost completely) in the right graph.}\label{fig:choppingUp}
\end{figure}
In particular all copies are drawn at the same horizontal level as $v$.
There is at most a single edge emanating from $v$ downwards (connecting $v$ to its father); this edge will never be separated from the original $v$.
Some children of $v$, however, may disconnect from $v$ and connect instead to a copy of $v$.
These copies become the roots of a new component.
Thus the vertex $v$ remains filled or unfilled, but all new copies of $v$ will be unfilled, and they are the lowest point of their connected component (in the new graph). Splitting is applied only to vertices with degree at least two (we do not split leaves or roots with only one child).

If $\Gamma'$ is obtained by chopping up $\Gamma$, then we indicate this fact by $\Gamma\prec \Gamma'$. 
We clearly have transitivity, i.e. if $\Gamma,\Gamma',\Gamma''\in\mathcal F_k$ with $\Gamma\prec \Gamma'$ and $\Gamma'\prec \Gamma''$, then
\be\label{lem:transitivity}\Gamma\prec \Gamma''.\ee
Now we define a particularly useful subset of possible splittings:

\begin{defi}
	Fix any vertex $v$ of $\Gamma \in \mathcal F_k$ with number of children $c(v)\geq 1$. 
A splitting at the vertex $v$ is called \emph{complete} if it yields $c(v)$ copies.

A splitting at $v$ which yields $c(v)-1$ copies is called \emph{almost complete}. In this latter case, if the edge connecting $v$ to its remaining child was the leftmost\footnote{Since the forest $\Gamma$ is drawn in the plane in a specific way, right and left are meaningful concepts.} one (out of all the edges connecting $v$ to its children), we call the splitting \emph{leftmost} and if it was the rightmost one, we call the splitting \emph{rightmost}.
\end{defi}

We call a forest $\Gamma$ {\it linear}, if graph-theoretically it is a union of paths, i.e. the degree of every vertex is at most two. 
%
	Notice that we obtain a linear $\Gamma'$ in chopping up $\Gamma$, if we split every vertex $v$ of $\Gamma$ with $c(v)\geq 1$ either \emph{completely} or \emph{almost completely}.
	

\subsubsection{Monotonicity of $val$ along chopping}
We now extend the previously introduced concept of a {\it value} from trees (\eqref{def:valNj}, \eqref{def:valOfTree}) to any $\Gamma\in\mathcal F_k$.


\begin{defi}
Every vertex $v$ gets a label $x_v\in\{1,\dots,N\}$ 
and labels assigned to non-root vertices are summed up, while we take the maximum over labels assigned to roots:
\be\label{def:val}
   val\ofN(\Gamma) \deq \Big(\prod_{v\in R(\Gamma)} \max_{x_v}\Big) 
   \Big(\prod_{v\in V(\Gamma)\setminus R(\Gamma)}  \sum_{x_v=1}^N \Big)\Big[ \prod_{e\in E(\Gamma)} S_{x_{e_-} x_{e_+}} \Big].
\ee
where $R(\Gamma)\subseteq V(\Gamma)$ is the set of roots.
\end{defi}

Note that this coincides with \eqref{def:valOfTree} fir trees, i.e. when there is only one root.
See Figure \ref{fig:choppedUpTreesAndTheirValues} for some chopped up trees and their values.

\begin{figure}[h]
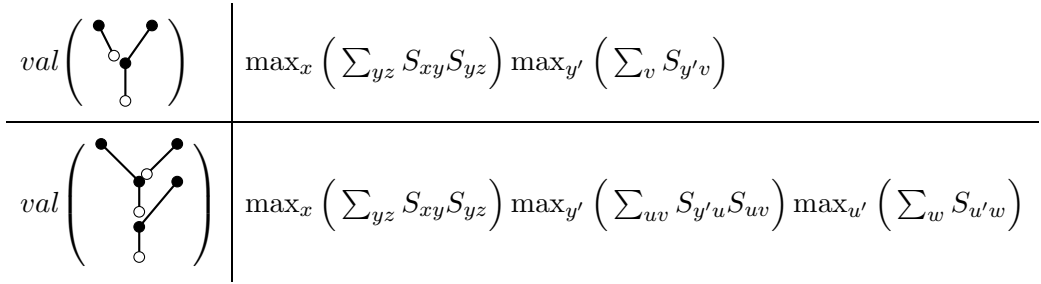
\begin{tabular}{l|l}
%
%
			
	\trimbox{0cm -0.2cm 0cm -0.2cm}{$val\ofN\Bigg($ \tikz[baseline=2.5ex]{
			\draw[thick] (0,0) -- (0,0.5);
			\draw[thick] (-0.15,0.6) -- (-0.35,1);
			\draw[thick] (0,0.5) -- (0.35,1);
			\draw[fill=white] (0,0) circle (2pt);
			\draw[fill=black] (0,0.5) circle (2pt);
			\draw[fill=white] (-0.15,0.6) circle (2pt);
			\draw[fill=black] (0.35,1) circle (2pt);
			\draw[fill=black] (-0.35,1) circle (2pt);
			} $\Bigg)$} & $\max_{x}\Big(\sum_{yz} S_{xy}S_{yz}\Big)\max_{y'}\Big(\sum_{v}S_{y'v}\Big)$\\\hline
			
	\trimbox{0cm -0.2cm 0cm -0.2cm}{$val\ofN\Biggg($ \tikz[baseline=3.6ex]{
			\draw[thick] (0,0) -- (0,0.4) -- (0.5,1);
			\draw[thick] (0,0.55) -- (0,1) -- (-0.5,1.5);
			\draw[thick] (0.1,1.1) -- (0.5,1.5);
			\draw[fill=black] (0,0.4) circle (2pt);
			\draw[fill=black] (0.5,1) circle (2pt);
			\draw[fill=black] (0,1) circle (2pt);
			\draw[fill=black] (-0.5,1.5) circle (2pt);
			\draw[fill=black] (0.5,1.5) circle (2pt);
			\draw[fill=white] (0,0) circle (2pt);
			\draw[fill=white] (0,0.6) circle (2pt);
			\draw[fill=white] (0.1,1.1) circle (2pt);
			} $\Biggg)$} & $\max_{x}\Big(\sum_{yz} S_{xy}S_{yz}\Big)\max_{y'}\Big(\sum_{uv} S_{y'u}S_{uv}\Big)\max_{u'}\Big(\sum_{w}S_{u'w}\Big)$
\end{tabular}
\caption{Chopped-up trees and their values.}\label{fig:choppedUpTreesAndTheirValues}
\end{figure}


\begin{lem}\label{monotonic}
	The value function is monotonic along the chopping up operation:
	\begin{equation}\label{eq:monotonic}
	    val\ofN (\Gamma) \le val\ofN (\Gamma') \qquad \text{if} \qquad \Gamma\prec \Gamma'.
	\end{equation}
\end{lem}

\begin{proof}
The proof of this statement is an easy induction on subsequent splitting of vertices. It is based upon the trivial estimate \be\label{eq:trivialBoundForSplittingUpS}
		\max_a\sum_{b,c}S_{ab}S_{bc}\leq \max_a\sum_b S_{ab}\max_{b'}\sum_c S_{b'c} = \|S\|^2,
	\ee
which, in our graphical language, can also be written as:

\begin{figure*}[h]
	$val\ofN\Bigg($ \tikz[baseline=2.5ex]{
		\draw[thick] (0,0) -- (0,1);
		\draw[fill=white] (0,0) circle (2pt);
		\draw[fill=black] (0,0.5) circle (2pt);
		\draw[fill=black] (0,1) circle (2pt);} $\Bigg) \leq val\ofN\Bigg($ \tikz[baseline=2.5ex]{
		\draw[thick] (0,0) -- (0,0.35);
		\draw[thick] (0,0.65) -- (0,1);
		\draw[fill=white] (0,0) circle (2pt);
		\draw[fill=black] (0,0.35) circle (2pt);
		\draw[fill=white] (0,0.65) circle (2pt);
		\draw[fill=black] (0,1) circle (2pt);} $\Bigg)$

\end{figure*}
\noindent Here the original vertex with label $b$ was split, the copy received a new label $b'$. For more complicated graphs the proof is similar.
\end{proof}


\subsubsection{Bounding $val\ofN(\Gamma)$ using the chopping process}
Now we will fix $N,k\in\N$, $\Gamma\in\mathcal T_k$ and its corresponding Dyck path $\pi=\pi(\Gamma)$. 
We will chop up $\Gamma$, i.e. construct a linear chopped-up tree $\Gamma'\in\mathcal F_k$ with $\Gamma\prec\Gamma'$ with the minimal amount of chopping.

%

As an example, Figure \ref{fig:splittingUpGammaInLeftandRightmostWay} shows the leftmost and rightmost almost complete splittings. They give rise to the bounds \eqref{eq:exampleBndUp} and \eqref{eq:exampleBndDown}, respectively.

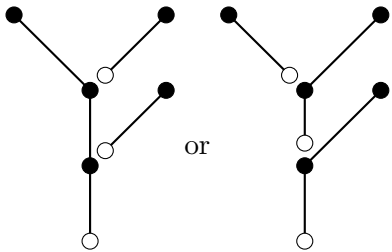
\begin{figure}[h]
	\begin{tikzpicture}[baseline=7ex]
		\draw[thick] (0,0) -- (0,1) -- (0,2) -- (-1,3);
		\draw[thick] (0.2,2.2) -- (1,3);
		\draw[thick] (0.2,1.2) -- (1,2);		
		\draw[fill=white] (0,0) circle (3pt);
		\draw[fill=black] (0,1) circle (3pt);
		\draw[fill=white] (0.2,1.2) circle (3pt);
		\draw[fill=black] (0,2) circle (3pt);
		\draw[fill=white] (0.2,2.2) circle (3pt);
		\draw[fill=black] (-1,3) circle (3pt);
		\draw[fill=black] (1,3) circle (3pt);
		\draw[fill=black] (1,2) circle (3pt);
	\end{tikzpicture}
	or
	\begin{tikzpicture}[baseline=7ex]
		\draw[thick] (0,0) -- (0,1) -- (1,2);
		\draw[thick] (0,1.2) -- (0,2) -- (1,3);
		\draw[thick] (-0.2,2.2) -- (-1,3);
		\draw[fill=white] (0,0) circle (3pt);
		\draw[fill=black] (0,1) circle (3pt);		
		\draw[fill=white] (0,1.3) circle (3pt);
		\draw[fill=black] (0,2) circle (3pt);
		\draw[fill=white] (-0.2,2.2) circle (3pt);
		\draw[fill=black] (-1,3) circle (3pt);
		\draw[fill=black] (1,3) circle (3pt);
		\draw[fill=black] (1,2) circle (3pt);
	\end{tikzpicture}
	\caption{Splitting up $\Gamma$ from Figure \ref{fig:labelledTreeAndDyckPathRepresentation} in the leftmost and the rightmost way.}\label{fig:splittingUpGammaInLeftandRightmostWay}
\end{figure}


\begin{rem}[Structure of bounds from the Dyck path representation of $\Gamma$]\label{rem:gainOnUpOrDownRuns}
Notice that with the leftmost choice, we made a gain on the  monotonically increasing parts (consecutive up-runs) of the corresponding Dyck path $\pi$, 
while with the rightmost choice we gained on the monotonically decreasing parts (consecutive down-runs). Here ``gain" means that we did not chop up the corresponding monotonic segments into pieces of length one; this allowed us to use the norms of higher powers of $S$ instead of trivially estimating them by higher powers of $\|S\|$.
\end{rem}

In what follows we want to quantify this gain. Recall the definition of $z_j$ from \eqref{eq:def_zj}.
For any sequence $T=(T_1, T_2,...,T_J)$  of nonnegative integers we set the notation
$$
   z^T\deq \prod_{j=1}^J z_j^{T_{j}}.
$$
The upper cutoff $J$ is a fixed parameter in Theorem \ref{thm:main}.
Note that $T_1$ does not influence $z^T$ since $z_1=1$. 

Define for any fixed $J\in\N$ and any path $\pi$ (i.e. any sequence $\pi = (\pi(i))_{i\in [a,b]}$ with $\pi(i)\in\N, |\pi(i+1)-\pi(i)|=1$) the $J$-tuples $U(\pi)$ and $D(\pi)$ by
\be\label{def:UpiDpi}
	U(\pi)_j\deq \#\{\text{up-runs of length }j\}
	\quad\text{and}\quad
	D(\pi)_j\deq \#\{\text{down-runs of length }j\}
\ee
for $j=1,2,\dots,J$. 
%
%
The observation from Remark \ref{rem:gainOnUpOrDownRuns} proves the following:

\begin{lem}\label{lem:boundByUpRunsOrByDownRuns}
Let $\Gamma\in\mathcal T_k$ and let $\pi=\pi(\Gamma)$ be the Dyck path corresponding to $\Gamma$.
With the rightmost choice we get 
\be\label{U}
   \frac{val(\Gamma)}{\|S\|^k} \le z^{U(\pi)}, 
\ee
while the leftmost choice gives
\be\label{D}
   \frac{val(\Gamma)}{\|S\|^k} \le z^{D(\pi)},
\ee
\hfill\qed
\end{lem}

In our concrete case (compare Figures \ref{fig:labelledTreeAndDyckPathRepresentation} and \ref{fig:splittingUpGammaInLeftandRightmostWay}) we have $U=(2,0,1,0, \dots)$ and $D=(1,2,0,\dots)$. 

\subsubsection{Introducing our choice of chopping up $\Gamma$}

We will need a mixture of the two estimates \eqref{U} and \eqref{D}, and it will be more convenient to work with the Dyck path $\pi=\pi(\Gamma)\in D_{2k}$ corresponding to $\Gamma\in\mathcal T_k$. Namely, above a certain threshold height $\ell$ we will 
to follow the rightmost choice, below that level the leftmost choice (see Figure \ref{fig:thresholdIllustration} for a naive sketch).

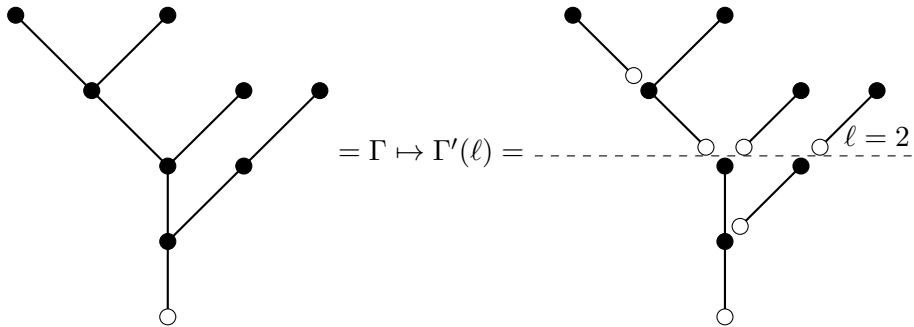
\begin{figure}[h]
	\begin{tikzpicture}[baseline=12.5ex]
		\draw[thick] (0,0) -- (0,1) -- (0,2) -- (-1,3) -- (-2,4);
		\draw[thick] (-1,3) -- (0,4);
		\draw[thick] (0,1) -- (1,2) -- (2,3);
		\draw[thick] (0,2) -- (1,3);
		\draw[fill=white] (0,0) circle (3pt);
		\draw[fill=black] (0,1) circle (3pt);
		\draw[fill=black] (0,2) circle (3pt);
		\draw[fill=black] (1,2) circle (3pt);
		\draw[fill=black] (2,3) circle (3pt);
		\draw[fill=black] (1,3) circle (3pt);
		\draw[fill=black] (-1,3) circle (3pt);
		\draw[fill=black] (-2,4) circle (3pt);
		\draw[fill=black] (0,4) circle (3pt);
	\end{tikzpicture}
	$=\Gamma\mapsto\Gamma'(\ell)=$
	\begin{tikzpicture}[baseline=12.5ex]
		\draw[thick] (0,0) -- (0,1) -- (0,2);
		\draw[thick] (0.2,1.2) -- (1,2);
		\draw[thick] (1.25,2.25) -- (2,3);
		\draw[thick] (0.25,2.25) -- (1,3);
		\draw[thick] (-0.25,2.25) -- (-1,3) -- (0,4);
		\draw[thick] (-1.2,3.2) -- (-2,4);
		\draw[dashed] (-2.5,2.135) -- (2.5,2.135);
		\draw[fill=white] (0,0) circle (3pt);
		\draw[fill=black] (0,1) circle (3pt);
		\draw[fill=white] (0.2,1.2) circle (3pt);
		\draw[fill=black] (0,2) circle (3pt);
		\draw[fill=white] (-0.25,2.25) circle (3pt);
		\draw[fill=white] (0.25, 2.25) circle (3pt);
		\draw[fill=black] (1,2) circle (3pt);
		\draw[fill=white] (1.25,2.25) circle (3pt);
		\draw[fill=black] (2,3) circle (3pt);
		\draw[fill=black] (1,3) circle (3pt);
		\draw[fill=black] (-1,3) circle (3pt);
		\draw[fill=white] (-1.2,3.2) circle (3pt);
		\draw[fill=black] (-2,4) circle (3pt);
		\draw[fill=black] (0,4) circle (3pt);
		\node at (2,2.4) {$\ell=2$};
	\end{tikzpicture}
	
	\caption{Graph splitting at level $\ell=2$ (Above $\ell$: rightmost splitting, below $\ell$: leftmost splitting, at $\ell$: complete splitting.)}\label{fig:thresholdIllustration}
\end{figure}

Moreover, this choice will be determined not by the actual height of the vertex, but by the height of $\pi$ at certain coarse-grained cutoff times in order to avoid that the rightmost and leftmost choices alternate too often.
These requirements necessitate a slightly more refined construction.

Choose a (small) parameter $\e$ and define the sequence of {\it cutoff times}
$$
   t_j \deq \lfloor 2k\e j\rfloor, \qquad j=0,1,\ldots, 1/\e
$$
(we assume that $1/\e$ is an integer). These cutoff times naturally split any path $\pi$ into $1/\e$ segments of equal\footnote{Equal up to $\pm 1$, which will not matter as $k\to\infty$. Henceforth we will assume that $t_j\eqbydef \lfloor 2k\e j\rfloor = 2k\e j$.} length, i.e.
$$
  \pi = \pi_{[0,2k]} = \bigcup_{j=0}^{1/\e-1} \pi_j, \qquad \pi_j := \pi_{[t_{j}, t_{j+1}]}.
$$
Even though $\pi_j$ is defined as the restriction of $\pi$, with a slight abuse of notation we will shift its argument starting at $0$, i.e. strictly speaking $\pi_j=\pi_{[t_j,t_{j+1}]}\circ s_{-t_j}$ where $s_{-t_j}(t) \deq t-t_j$ is the shift operator.
The height $\pi(t_j)$ of the beginning of each segment will be called the $j$-th {\it cutoff height}. 
For every integer $i\in [0,2k]$ there is a unique $j$ such that $i\in [t_j, t_{j+1})$ and the {\it cutoff height of $i$} is defined to be $\pi(t_j)$, i.e. the cutoff height of any index $i$ is determined by the initial point of its segment.

Given the parameters $\e>0$,  $\ell>1$, $\ell\in \N$, and given a $\Gamma\in\mathcal T_k$, and hence the corresponding Dyck path $\pi\in D_{2k}$, we now define a specific chopped-up graph $\Gamma'(\e, \ell)$ with $\Gamma\prec \Gamma'(\e, \ell)$.

First we define a subset $R$ of the (integer) time variables in $[0,2k]$ as follows: 
$$
   R\deq\bigcup_{j=0}^{1/\e-1} R_j,
$$
with
$$
	R_j \deq \begin{cases} \{ i \in [t_j, t_{j+1})\; : \; \pi(i)<\pi(i+1)\}   &  \text{if}  \quad  \pi (t_j)\le 2k\e (\ell -1)  \\
   \{ i \in [t_j, t_{j+1})\; : \; \pi(i)>\pi(i+1)\}  & \text{if}  \quad \pi (t_j)\ge 2k\e(\ell +1), \\
 \emptyset & \text{if}  \quad 2k\e(\ell-1) <\pi (t_j)< 2k\e(\ell +1).
 \end{cases}
$$
The set $R$ contains those times $i\in\N$ when the path goes upwards whenever its  cutoff height is below the lower threshold $2k\e (\ell -1)$ as well as those times when the path goes downwards if the cutoff height is above the upper threshold $2k\e (\ell +1)$.
For any $i\in R$, we \emph{tag} the edges between $(i,\pi(i))$ and $(i+1, \pi(i+1))$ of the Dyck path and draw them bold (see Figure \ref{fig:exampleGraphicForR}).

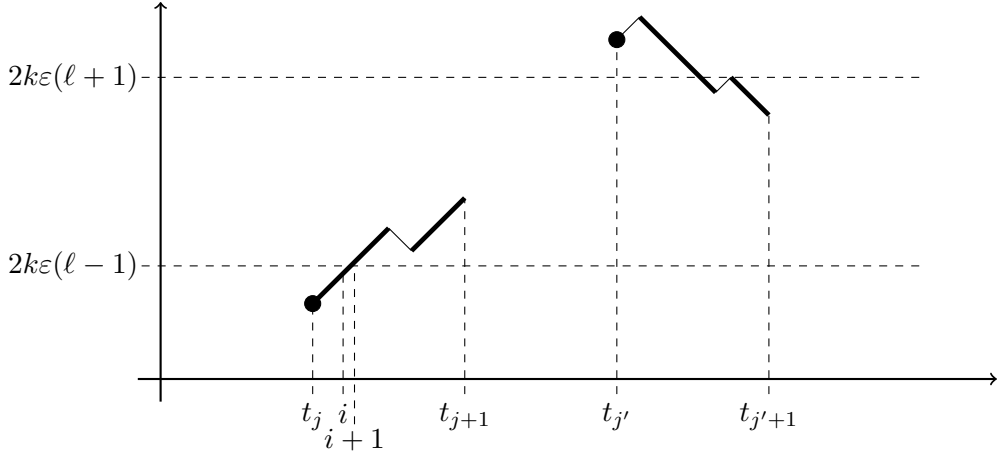
\begin{figure}[h]
	\begin{tikzpicture}[baseline=8ex]
		\draw[fill=black] (2,1) circle (3pt);
		\draw[fill=black] (6,4.5) circle (3pt);
		\node at (-1.15,1.5) {$2k\e(\ell-1)$};
		\node at (-1.15,4) {$2k\e(\ell+1)$};
		\node at (2,-0.5) {$t_j$};
		\node at (2.4,-0.45) {$i$};
		\node at (2.55,-0.8) {$i+1$};
		\node at (4,-0.5) {$t_{j+1}$};
		\node at (6,-0.5) {$t_{j'}$};
		\node at (8,-0.5) {$t_{j'+1}$};
		\draw[->,thick] (-0.3,0)--(11,0);
		\draw[->,thick] (0,-0.3)--(0,5);
		\draw[dashed] (-0.25,1.5) -- (10,1.5);
		\draw[dashed] (-0.25,4) -- (10,4);
		\draw[dashed] (2,-0.2) -- (2,1);
		\draw[dashed] (4,-0.2) -- (4,2.4);
		\draw[dashed] (2.4,-0.2) -- (2.4,1.4);
		\draw[dashed] (2.55,-0.6) -- (2.55,1.55);
		\draw[dashed] (6,-0.2) -- (6,4.5);
		\draw[dashed] (8,-0.2) -- (8,3.5);
		\draw[line width=1.8pt] (2,1) -- (3,2);
		\draw[thin] (3,2) -- (3.3,1.7);
		\draw[line width=1.8pt] (3.3,1.7) -- (4,2.4);
		\draw[thin] (6,4.5) -- (6.3,4.8);
		\draw[line width=1.8pt] (6.3,4.8) -- (7.3,3.8);
		\draw[thin] (7.3,3.8) -- (7.5,4);
		\draw[line width=1.8pt] (7.5,4) -- (8,3.5);

		
	\end{tikzpicture}
	\caption{Sketch for tagged (bold) edges defined by $R$.}\label{fig:exampleGraphicForR}
\end{figure}

Moreover, define
$$
   P\deq\{ i\in R\; : \; (\pi(i+1)-\pi(i))(\pi(i)-\pi(i-1))>0\},
$$
i.e. these are the indices $i\in R$ such that $(i,\pi(i))$ is in the middle of a monotonic segment of length at least 2. In particular, we have the following property:

\begin{lem}[Property {\bf P}]
If  $i\in P$, then both edges of the Dyck path adjacent to the point $(i,\pi(i))$ are tagged.
Moreover, if  $i\in [t_j, t_{j+1})$ for some $j$, then exactly one of the following two options holds:
\begin{itemize}
	\item either: $\pi(t_j) \le 2k\e(\ell -1)$ and $\pi(i-1)<\pi(i)<\pi(i+1)$ 
	\item or: $\pi(t_j)\ge 2k\e (\ell +1)$ and $\pi(i-1)>\pi(i)>\pi(i+1)$. \hfill\qed
\end{itemize}
\end{lem}
This construction (together with the fact that the maximal height difference within each segment $\pi_j$ is at most $2k\e$) implies the following observations:
\begin{itemize}
	\item[(i)] Below a security layer of width $4k\e$ around the fixed level $2k\e\ell$ all up-runs of $\pi$ are tagged, above the security layer all down-runs of $\pi$ are tagged.
	\item[(ii)] At any given level all tagged edges are of the same type (upward or downward).
	\item[(iii)] The choice whether the up-runs or the down-runs are tagged is decided at the cutoff times  $t_j$ and this choice is valid for the entire path segment $\pi_j$.
\end{itemize}

Now we are ready to define the chopped up graph $\Gamma'(\e, \ell)$ that we will actually use.
 
\begin{defi}[Definition of $\Gamma'(\e, \ell)$]\label{def:chopping-up}
Fix $k\in\N$, $\Gamma\in\mathcal T_k$, and let $\pi=\pi(\Gamma)$ be the corresponding Dyck path. We define $\Gamma'(\e, \ell)$ by the following procedure that determines how we split the vertices $V(\Gamma)$:

We walk around $\Gamma$ starting from the root in clockwise direction and we successively mark all edges to be split either completely or almost completely in the leftmost or the rightmost way (but we do not split them yet).
The marking is determined by the following rules:

\underline{\emph{Step 1:}}
We mark the root for leftmost splitting.	

\underline{\emph{Step 2:}}
Now consider the $i$-th step (for $i>0$, the root has been dealt with) and fix $j$ such that $i\in [t_j, t_{j+1})$. Let $v(i) \in V(\Gamma)$ denote the vertex reached at the $i$-th step.
\begin{enumerate}[(a)]
	\item\label{it:tjSplitBot} If $i=t_j$, $\pi(t_j) \le 2k\e(\ell-1)$, and $v(i)$ is visited for the first time\footnote{I.e. there is no $j<i$ with $v(j)=v(i)$.}, then we mark the vertex $v(i)$ for complete splitting.
	\item\label{it:tjSplitTop} If $i=t_j$, $\pi(t_j) \geq 2k\e(\ell+1)$, and $v(i)$ is visited for the last time, then we also mark $v(i)$ for complete splitting.
	\item\label{it:PSplitting} If $i\neq t_j$, $i\in P$, then we mark $v(i)$ for almost complete splitting either in the leftmost or the rightmost way, depending on whether $\pi(t_j) \le 2k\e(\ell-1)$ or $\pi(t_j)\ge 2k\e(\ell+1)$, respectively.
\end{enumerate}

\underline{\emph{Step 3:}} Consider all vertices that were left unmarked\footnote{Note that, in particular, this was the case for $i\in[t_j,t_{j+1})$ such that $2k\e(\ell-1)<\pi (t_j)< 2k\e(\ell+1)$.} in Step 1 and mark them for complete splitting.
Finally, we perform the prescribed splittings.
\end{defi}


\begin{lem}
	The procedure described in Definition \ref{def:chopping-up} is well-defined, i.e. every vertex receives an unambiguous marking.
\end{lem}
\begin{proof}
The root is always marked for leftmost splitting by Step 1. 
Since the root is visited the first time at $i=0$, rule \eqref{it:tjSplitBot} together with $i>0$ does not mark the root. The conditions of rule \eqref{it:tjSplitTop} also exclude the root since for the root $\pi(i)=0$. Finally, rule \eqref{it:PSplitting} applies only to vertices in the middle of a monotonic segment ($i\in P$), hence it does not apply to the root either, thus Step 1 is not in conflict with Step 2.

Now consider the vertices visited at times $i=t_j$ for some $j$. It is easy to see from the definition of $P$ there is no other time $i'\neq i$ with $v(i')=v(t_j)$ s.t. \eqref{it:PSplitting} marks $v(i')$, hence there is no conflict between \eqref{it:PSplitting} and \eqref{it:tjSplitBot}, \eqref{it:tjSplitTop}.
There is no conflict between \eqref{it:tjSplitBot} and \eqref{it:tjSplitTop} due to the mutually exclusive conditions on $\pi(t_j)$.

It remains to show that rule \eqref{it:PSplitting} is applied to the same vertex $v(i)$ at most once.
When walking around $\Gamma$, the same vertex $v\in V(\Gamma)$ is visited several (even number of) times, say $v(i_1)= v(i_2) =\ldots = v(i_{2m})=v$. 
However, we claim that only at most one of the time indices $i_1, i_2, \ldots, i_{2m}$ can be in $P$ (and if $v$ is the root, then clearly none can be in $P$); in other words, the above procedure triggers a splitting  of $v$ at most at one of the times $i_1, i_2, \ldots, i_{2m}$.
Assuming to the contrary that there exist $i, i'\in P$ with $v(i)=v(i')$ and $i\ne i'$, then by Property P both $(i, \pi(i))$ and $(i', \pi(i'))$ are joining two marked edges of $\pi$ of the same monotonicity type.
By the observation (ii) above, either both are increasing:
$\pi(i-1)<\pi(i)<\pi(i+1)$ and $\pi(i'-1)<\pi(i')<\pi(i'+1)$, or both are decreasing: $\pi(i-1)>\pi(i)>\pi(i+1)$ and $\pi(i'-1)>\pi(i')>\pi(i'+1)$. 
However, the construction of the graph $\Gamma$ from $\pi$ excludes $v(i)=v(i')$ in both cases, which is a contradiction, proving the original claim.
In particular, to every $v$ that is split by rule \eqref{it:PSplitting} (but not by \eqref{it:tjSplitBot} or \eqref{it:tjSplitTop}) along the procedure above, there is a unique time $i=i_v$, when it was split.
\end{proof}
One may also arrive at $\Gamma'(\e,\ell)$ as follows. Split the root almost completely in the leftmost way. Consider any $i>0$ and fix $j=j(i)$ such that $i\in [t_j,t_{j+1})$. We first split almost completely all those vertices $v(i)\in V(\Gamma)$ where $h(t_j)\not\in [2k\e(\ell-1), 2k\e(\ell+1)]$ that are either in an up-run, i.e. $\pi(i-1)<\pi(i)<\pi(i+1)$, in case $\pi(t_j)\leq 2k\e (\ell-1)$ or in a down-run, i.e. $\pi(i-1)>\pi(i)>\pi(i+1)$, in case $\pi(t_j) \geq 2k\e(\ell+1)$.
The corresponding splitting is in the leftmost or the rightmost way, respectively.
Next we split some of these vertices even further, namely those vertices $v(t_j)$ that are in an up-run (if $\pi(t_j)\leq 2k\e(\ell-1)$) and those in a down-run (if $\pi(t_j)\geq 2k\e(\ell+1)$) we split completely.
Finally, in the last step, all unsplit vertices are split completely.


\begin{lem}
Fix $\e,\ell$. Let $\pi\in D_{2k}$ and consider its $j$-th subpath $\pi_j\eqbydef\pi_{[t_j,t_{j+1}]}$. Set
\be\label{eq:defTtilde}
	\wt T_{\ell,\e}(\pi)\deq\sum_{j=0}^{1/\e-1}
		\Big[ U(\pi_j) \cdot {\bf 1}( \pi( t_j)\le 2k\e(\ell-1) ) + D(\pi_j) \cdot {\bf 1}( \pi( t_{j})\ge 2k\e(\ell+1)) \Big].
\ee
Let $\Gamma'(\e,\ell)$ be the chopping-up of the tree $\Gamma=\Gamma(\pi)$ given in Definition \ref{def:chopping-up}. Then we have
\be\label{TT}
	\frac{val(\Gamma)}{\|S\|^k}
	\le z^{\wt T_{\ell,\e}(\pi)}.
\ee
\end{lem}
\begin{proof}
Recalling that by Definition \ref{def:chopping-up} $\Gamma'(\e,\ell)$ is a set of \emph{linear} trees, we have
\be\label{eq:forestbnd}
\frac{val(\Gamma)}{\|S\|^k}
	\leqby{\eqref{eq:monotonic}} \frac{val(\Gamma'(\e, \ell))}{\|S\|^k}
	\leq \prod_{l\in\Gamma'(\e,\ell)} z_{|E(l)|},
\ee
where $|E(l)|$ denotes the number of edges in the linear tree $l$. The second inequality in \eqref{eq:forestbnd} holds since by the definition in \eqref{def:val} the value of a forest is just the product over values of its (tree) components and the value of a linear tree of length $n$ is estimated by $\|S^n\| \eqbydef z_n \|S\|^n$.

The r.h.s. of \eqref{eq:forestbnd} is equal to $z^{\wt T_{\ell,\e}(\pi)}$. To see this, note that rules \eqref{it:tjSplitBot} and \eqref{it:tjSplitTop} allow us to consider the subpaths $\pi_j$ independently of each other (like in the definition of $\wt T_{\ell,\e}$) and rule \eqref{it:PSplitting} ensures we do not ``overcount".
Note that every edge of $\Gamma'(\e,\ell)$ gives rise to exactly two edges of $\pi$ (at the same height). By not overcounting we mean that we need to make sure to use at most one of these two edges in $\wt T_{\ell,\e}$. This is obvious since one of these edges is going up, one is going down and the characteristic functions in \eqref{eq:defTtilde} take only segments going up \emph{or} going down at any given height into account.

\end{proof}

Instead of working with $\wt T_{\ell,\e}$, we would prefer to work with something more tractable without technical restrictions of security layers, similar to what we sketched in Figure \ref{fig:thresholdIllustration}.
To this end we define  the $J$-sequences with a threshold at $2k\e\ell$ without security zone, i.e. we set 
$$
	T_{\ell,\e}(\pi)\deq\sum_{j=0}^{1/\e-1}  T_{\ell,\e}(\pi_j),\qquad
	T_{\ell,\e}(\pi_j)\deq U(\pi_j) \cdot {\bf 1}( \pi( t_j)\le 2k\e\ell ) + D(\pi_j) \cdot {\bf 1}( \pi( t_{j})>2k\e \ell ).
$$
We also define
$$
	\Delta_{\ell, \e}(\pi) \deq \# \{ j\in [0, 1/\e)\; : \;|\pi(t_j)- 2k\e\ell |\le 2k\e\},
$$
the number of cutoff times when the cutoff height is close to the threshold $2k\e\ell$.

We introduce the shorthand notation to denote the expectation w.r.t. the uniform measure on $\mathcal T_k$ $$\E f(\Gamma)\deq \frac{1}{|\mathcal T_k|}\sum_{\Gamma'\in\mathcal T_k}f(\Gamma'),$$ for any $f:\mathcal T_k\to\R$.
We use a similar convention for $f:D_{2k}\to\R$, using the bijection between $D_{2k}$ and $\mathcal T_k$.

Now we quantify at what cost we can consider $z^{T_{\ell,\eps}}$ instead of $z^{\wt T_{\ell,\eps}}$ in \eqref{TT}.
Given $\Gamma\in\mathcal T_k$, its corresponding Dyck path $\pi\in D_{2k}$, $\e>0$, $\ell>1$ integer, we clearly have
\be\label{mainlem}
	\frac{val(\Gamma)}{\| S\|^k}\leq  \frac{ z^{T_{\ell,\e}(\pi)} }{  (\min_{j\le J} z_j^{1/j})^{\Delta_{\ell,\e}(\pi)2k\e} }.
\ee

%
For any fixed $\ell$ there are some $\Gamma$ for which the bound given by \eqref{mainlem} is very bad.
Namely, if most $\pi(t_j)$ are close to $2k\e\ell$, i.e. the path spends a lot of time in the security layer, then $\Delta_{\ell,\e}$ is large and the estimate \eqref{mainlem} is weak.
To prevent this, we will choose the security layer depending on the path in a coarse-grained fashion in the next lemma.



\begin{lem}\label{lem:with2k}
Fix two parameters $L$ and $M$, then
\be\label{with2k}
	\big[ \E \; val(\Gamma)\big]^{1/2k}\le \frac{L^{1/2k}}{  (\min_{j\le J} z_j^{1/j})^{1/L} } \| S\|^{1/2}\cdot
	\max_{m\le L} \big[ \E\; z^{T_{M+2m,\e}(\pi)} \big]^{1/2k}.
\ee
\end{lem}

\begin{proof}
By the pigeonhole principle, for any fixed $\pi\in D_{2k}$, there exists an $\ell=\ell(\pi)$ of the form $\ell = M+2m$ with $m\in \{1, 2, \ldots, L\}$
such that
$$
	\Delta_{\ell,\e}(\pi) \le \frac{1}{\e L}.
$$
Choosing this $\ell=\ell(\pi)$ in the estimate \eqref{mainlem}, we obtain
$$
	\frac{val(\Gamma)}{\| S\|^k}\le  \frac{ z^{T_{\ell(\pi),\e}(\pi)} }{  (\min_{j\le J} z_j^{1/j})^{2k/L} }.
$$
Summing up for all possible values of $\ell(\pi)$, we obtain the following bound:
$$
    \frac{val(\Gamma)}{\| S\|^k}\le   \frac{1}{  (\min_{j\le J} z_j^{1/j})^{2k/L} } \sum_{m=1}^L  z^{T_{M+2m,\e}(\pi)} .
$$
Now we take expectation $\E$ 
and take the $2k$-th root, to finish the proof.
\end{proof} 

We will take the limits in the following order
$$
  \lim_{L,M}\lim_\e \limsup_k
$$
which makes the prefactor in \eqref{with2k} one, hence negligible.
So it is sufficient to estimate 
\be\label{eq:suffToEstimate_zTl}
	\lim_\e \limsup_k \E \; z^{T_{\ell,\e}(\pi)} = \lim_\e \limsup_k \E \; \prod_{j=0}^{1/\e-1} z^{ T_{\ell,\e}(\pi_j)}
\ee
for $\ell = M+2m$ fixed.



\section{Distribution of Dyck Paths}
The uniform measure on the set of Dyck paths is equivalent to an inhomogeneous Markov chain that we describe now.

Fix $k$ and note that because of the boundary conditions no Dyck path can leave the triangle $$\Delta_k\deq \{(t,h)\in\N^2: h\leq t\text{ and } h\leq 2k-t\}.$$
Introduce furthermore
\begin{align*}
	\Delta_{k,\delta}\deq \Delta_k\cap\{(t,h):t+2k\e\leq 2k(1-\delta)\}\quad\quad
	& \Delta_{k,\delta}^{top}\deq \Delta_{k,\delta}\cap\{(t,h):h>2k\e\ell\}\\
	& \Delta_{k,\delta}^{bot}\deq \Delta_{k,\delta}\cap\{(t,h):h\leq 2k\e\ell\}
\end{align*}
as illustrated in Figure \ref{fig:transitionProbLandscape}.
For any $(t,h)\in\Delta_k$, let $$p_{t,h}\deq \P\left(\pi(t+1)=h+1|\pi(t)=h\right)$$
be the conditional probability, w.r.t. $\P$, the uniform measure on $D_{2k}$, that the path goes up at time $t$ and height $h$. In \cite{AS80} (Eq. (4)) it is shown that
\be\label{eq:arnoldTransitionProbabilities}
	p_{t,h}=\frac{1}{2}\frac{h+2}{h+1}\frac{2k-t-h}{2k-t}.
\ee

The Markov property allows us to consider parts of the path (of length $2k\e$, say) separately. We will encode these subpaths by their increments. To this end introduce $\Omega = \{+1,-1\}^{2k\e}$ and equip it with the natural $\sigma$-algebra $\mathcal F = 2^\Omega$. 

Fix $(t,h)\in\Delta_k$ such that $t+2k\e\leq 2k$, and an $i\in[0,2k\e]$ integer.
Define for $\omega=(\omega(i))_{i=1}^{2k\e}$ the \emph{absolute height at relative time} $i$ as $$h_i^{abs}=h_i^{abs}(\omega)\deq h+\sum_{j=1}^{i} \omega(j),$$ and the \emph{absolute time} $t_i^{abs} \deq t+i$. 
Furthermore, for $i\in[0,2k\e)$ let
\begin{align*}
	p_{i+1}=p_{i+1}(\omega)\deq \frac{1}{2}\frac{h_i^{abs}+2}{h_i^{abs}+1}\frac{2k-t_i^{abs}-h_i^{abs}}{2k-t_i^{abs}},\quad \quad
&p_{i+1}^{top}= p_{i+1}^{top}(\omega)\deq\frac{1}{2}\frac{2k-t_i^{abs}-h_i^{abs}}{2k-t_i^{abs}},\\
&p_{i+1}^{bot}= p_{i+1}^{bot}(\omega)\deq \frac{1}{2}\frac{h_i^{abs}+2}{h_i^{abs}+1},
\end{align*}
as well as 
\begin{align*}
	\Pi_{t,h}\deq \Pi_{t,h}^{top}\cap\Pi_{t,h}^{bot},\quad\quad
	&\Pi_{t,h}^{top}\deq \{\omega\in\Omega : h_i^{abs}(\omega)\leq 2k-t_i^{abs},\;\forall i\in[0,2k\e)\},\\
	&\Pi_{t,h}^{bot}\deq \{\omega\in\Omega : h_i^{abs}(\omega)\geq 0,\;\forall i\in[0,2k\e)\}.
\end{align*}
The set $\Pi_{t,h}$ encodes the paths of length $2k\e$ that are legitimate continuations as a Dyck path starting from $(t,h)\in\Delta_k$.
On $\Pi_{t,h}\subseteq\Omega$ we define the probability measure $\P_{(t,h)}$ defined by
\be\label{eq:defP_th}
	\P_{(t,h)}(\omega) \deq \prod_{i=1}^{2k\e} (p_i(\omega))^{{\bf 1}(\omega(i)=1)}(1-p_i(\omega))^{{\bf 1}(\omega(i)=-1)}.
\ee
Similarly, on $\Pi_{t,h}^{top}$ and $\Pi_{t,h}^{bot}$ we define the probability measures $\P^{top}_{(t,h)}$ and $\P^{bot}_{(t,h)}$ with $p_i$ replaced by $p_i^{top}$ and $p_i^{bot}$ in \eqref{eq:defP_th}, respectively.
We naturally extend the probability measures $\P_{(t,h)}, \P_{(t,h)}^{top}, \P_{(t,h)}^{bot}$ to the entire measure space $(\Omega,\mathcal F)$ by setting them zero for $\omega$ not in $\Pi_{t,h},\Pi_{t,h}^{top},$ and $\Pi_{t,h}^{bot}$, respectively.

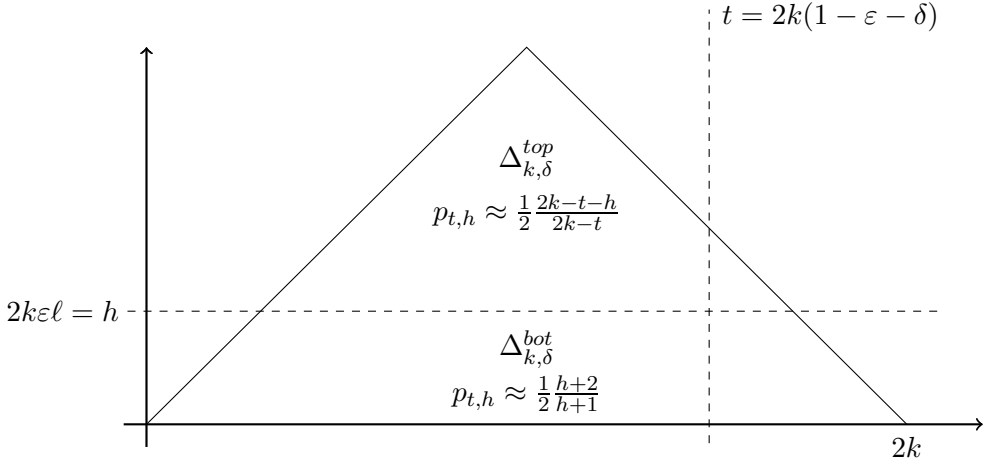
\begin{figure}[h]
	\begin{tikzpicture}[baseline=8ex]
		\node at (-1.1,1.5) {$2k\e\ell=h$};
		\node at (10,-0.3) {$2k$};
		\node at (9,5.4) {$t=2k(1-\e-\delta)$};
		\node at (5,2.8) {$p_{t,h}\approx \frac{1}{2}\frac{2k-t-h}{2k-t}$};
		\node at (5,3.5) {$\Delta_{k,\delta}^{top}$};
		\node at (5,0.4) {$p_{t,h}\approx \frac{1}{2}\frac{h+2}{h+1}$};
		\node at (5,1) {$\Delta_{k,\delta}^{bot}$};
		
		\draw[->,thick] (-0.3,0)--(11,0);
		\draw[->,thick] (0,-0.3)--(0,5);
		
		\draw[thin] (0,0) -- (5,5) -- (10,0);
		
		\draw[dashed] (-0.25,1.5) -- (10.5,1.5);
		\draw[dashed] (7.4,-0.25) -- (7.4,5.5);
		
		\draw[->,thin] (0,-0.3)--(0,5);
	\end{tikzpicture}
	\caption{Landscape of the transition probability.}\label{fig:transitionProbLandscape}
\end{figure}
The result \eqref{eq:arnoldTransitionProbabilities} shows that $p_i(\omega)$ is the transition probability (under the uniform distribution on $D_{2k}$) that the path starting at $(t,h)$ goes up in the $i$-th step after $t$, conditioned that it passed through at $(t,h)$ and its evolution between times $t$ and $t+i-1$ was given by $(\omega_1,\omega_2,\dots,\omega_{i-1})$.

Define the random variable $\pi^h:\omega\mapsto (h_i^{abs}(\omega))_{i=0}^{2k\e}$.\footnote{Since $\pi^h$ is a bijection for every fixed $h$ we will sometimes abuse notation to use $\omega$ and paths of length $2k\e$ (most prominently the subpaths $\pi_j\eqbydef\pi_{[t_j,t_{j+1}]}$) interchangeably.} Then clearly for every fixed $\omega\in\Omega$ we have that $\P_{(t,h)}(\omega)$ is the probability of $\{\tilde\pi_{[t,t+2k\e)}=\pi^h(\omega)\}$ for $\tilde\pi$ being sampled from the uniform distribution on $D_{2k}$, conditioned on $\{\tilde\pi(t)=h\}$, i.e.
$$\P_{(t,h)}(\omega) = \P\left(\tilde\pi_{[t,t+2k\e)}=\pi^h(\omega) \;|\; \tilde\pi(t)=h\right) \quad\quad\quad \forall\omega\in\Omega \text{ fixed}.$$
Writing $\E_{(t,h)}$ for the expectation under $\P_{(t,h)}$, we make use of the Markov property to write the last term on the r.h.s. of \eqref{with2k} as 

\be\label{eq:nestingOfExpectations}\E \; z^{T_{\ell,\e}(\pi)} =
\E \prod_{j=0}^{1/\e-1} z^{ T_{\ell,\e}(\pi_j)} =
\E_{(0,0)}\left[
	z^{T_{\ell,\e}(\pi_0)}
	\E_{(t_1,\pi_0(t_1))}\left[
		z^{T_{\ell,\e}(\pi_1)}\E_{(t_2,\pi_1(t_1))}[\dots]\right]\right],\ee
where $\ell = M+2m$. We also recall the convention that $\pi_j(i)=\pi(t_j+i)$ for all integers $i\in[0,2k\e]$, i.e. we start to count the time variable from zero for every resampled path.
In the following we will find an upper bound on $\E_{(t_j,h)}z^{T_{\ell,\e}(\pi_j)}$ independent of $j$ or $h$.
In fact, we bound $\E_{(t,h)}z^{T_{\ell,\e}(\pi)}$ for any path $\pi=\pi^h$ starting from $h$. Recall that $T_{\ell,\e}(\pi) = U(\pi)$ or $T_{\ell,\e}(\pi) = D(\pi)$, counting up-runs or down-runs, depending on whether $\pi^h(0)=h\leq 2k\e\ell$ or $h>2k\e\ell$, respectively.

\subsection{Bound by Simple Random Walk}
Let $\E_\mu$ be the expectation under a time-homogeneous random walk starting from $0$ and with probability $\mu$ of going up and $1-\mu$ of going down.
In particular, $\E_\frac{1}{2}$ refers to the expectation under the simple random walk.

\begin{lem}\label{lem:boundOnCondExp}
	Fix some (large) integers $\ell\geq 2$ and $k$, and (small) $\e,\delta>0$ such that $\e/\delta<(\ell+2)^{-1}$. Then for all $(t,h)\in\Delta_{k,\delta}$ we have
	$$\E_{(t,h)}z^{T_{\ell,\e}(\pi^h)} \le (1+\eta(k,\e,\delta))^{2k\e}\E_\frac{1}{2}z^{U(\pi^h)}$$
	with some error term $\eta(k,\e,\delta)=\eta_{t,h,\ell}(k,\e,\delta)$ that satisfies \be\label{eq:etaBound}
		|\eta(k,\e,\delta)|\leq \left((\ell+1)\frac{\e}{\delta} + \frac{1}{k\e}\right).
	\ee
	In particular, for any fixed $\ell$, we have the following limit uniformly in $t$ and $h$: $$\lim_{\e,\delta\to 0: \atop \e/\delta\to 0}\limsup_{k\to\infty}\eta(k,\e,\delta)=0.$$
\end{lem}

The intuition behind this estimate is that, for $h$ below some threshold $2k\e\ell$ and $t$ away from the endpoint $2k$ the probability measure $\P_{(t,h)}(\cdot)$ favours going up and hence we expect longer up-runs than in, say, a simple random walk. This is an effect of the repulsive boundary condition at $0$ that forces $\pi(i)\geq 0$ for all $i$. Longer up-runs clearly give us smaller $z^{U(\pi^h)}$ and in that regime $T_{\ell,\e}$ counts the up-runs. Similarly, for the region where $h>2k\e\ell$, the measure $\P_{(t,h)}(\cdot)$ favours going down due to the constraint $\pi(2k)=0$ at the endpoint and in this regime $T_{\ell,\e}$ counts the down-runs.
Note that the distribution of $U(\pi^h)$ and $D(\pi^h)$ are the same for the simple random walk, hence $\E_\frac{1}{2}z^{U(\pi^h)} = \E_\frac{1}{2}z^{D(\pi^h)}$.

To formalise this intuition we recall Holley's inequality from \cite{G06} (Theorems 2.1 and 2.6, as well as the remark after the statement of Theorem 2.1):

\begin{lem}[Holley's inequality]
	Let $\Omega \deq \{-1,1\}^E$ for some finite $E$ and $\mathcal F=2^\Omega$ be the discrete $\sigma$-algebra. Consider the partial order $\leq$ on $\Omega$, given by $\omega\leq\omega'$ iff $\omega(e)\leq\omega'(e)$ for all $e\in E$.
	
	Let $X$ be an increasing random variable from the measure space $(\Omega,\mathcal F)$ to $\R$, i.e. $X(\omega)\leq X(\omega')$ for any $\omega\leq \omega'\in\Omega$. Let $\mu_1,\mu_2$ be probability measures on $(\Omega,\mathcal F)$ satisfying
	\begin{enumerate}
			\item $\mu_1(\omega^e)\mu_2(\omega_e)\leq \mu_1(\omega_e)\mu_2(\omega^e)$, and
			\item $\mu(\omega_{ef})\mu(\omega^{ef})\geq \mu(\omega^e_f)\mu(\omega_e^f)$ for $\mu=\mu_1$ \underline{or} $\mu=\mu_2$,
	\end{enumerate}
	where $\omega^e$ and $\omega_e$ are defined by $\omega^e(i)=\omega_e(i)\deq\omega(i)$ if $i\neq e$ and $\omega^e(i)\deq 1$, $\omega_e(i)\deq -1$ if $i=e$. Furthermore we set\footnote{The order indicated by the bracket breaks the symmetry; one may have defined $\omega_f^e \deq (\omega^e)_f$. However, this notation only occurs in the combination $\mu(\omega_f^e)\mu(\omega_e^f)$ which is independent of this choice.} $\omega^e_f\deq (\omega_e)^f$, $\omega_{ef}\deq(\omega_e)_f$, and $\omega^{ef}\deq(\omega^e)^f$.
	Then we have $$\mu_1(X)\leq \mu_2(X).\quad\qed$$
\end{lem}
\noindent To apply Holley's inequality in the proof of Lemma \ref{lem:boundOnCondExp} we would like to approximate the measure $\P_{(t,h)}$ by a simpler Markov chain.
We define 
\be\widehat\P_{(t,h)}(\omega)\deq \begin{cases}
	\P_{(t,h)}^{top}(\omega), & \text{if } h > 2k\e\ell \\
	\P_{(t,h)}^{bot}(\omega), & \text{if } h\leq 2k\e\ell
\end{cases}\ee
and let $\widehat\E_{(t,h)}$ denote the expectation w.r.t. $\widehat\P_{(t,h)}$.

\begin{lem}\label{lem:timeHomApprox}
	Fix some (large) integers $\ell\geq 2$ and $k$, some (small) $\e,\delta>0$ such that $\e/\delta<(\ell+2)^{-1}$, and $(t,h)\in\Delta_{k,\delta}$. 
	Then we have \be\label{eq:timeHomApprox}
	\E_{(t,h)}z^{T_{\ell,\e}}\leq (1+\eta(k,\e,\delta))^{2k\e}\widehat\E_{(t,h)}z^{T_{\ell,\e}},
	\ee
	where the error term $\eta(k,\e,\delta) = \eta_{t,h,\ell}(k,\e,\delta)$ is chosen to be the same as in Lemma \ref{lem:boundOnCondExp}; in particular the bound \eqref{eq:etaBound} holds.
\end{lem}

The bound \eqref{eq:timeHomApprox} controls the measure $\P_{(t,h)}$ by a simpler measure $\widehat\P_{(t,h)}$. We could have made an approximation with a Markov chain with constant transition rates (on scale of $2k\eps$) in the regime that is far away from the boundary of $\Delta_k$.
This possibility is indicated in Figure \ref{fig:transitionProbLandscape}, but we will not need it in our proof.
%
\begin{proof}[Proof of Lemma \ref{lem:timeHomApprox}]
	We show that for all $(t,h)\in \Delta_{k,\delta}$ and all $\omega\in\Omega$, we have 
	\be\label{eq:PhatApprox}
		\P_{(t,h)}(\omega) \leq (1+\eta(k,\e,\delta))^{2k\e}\widehat\P_{(t,h)}(\omega),
	\ee
	then \eqref{eq:timeHomApprox} will follow.
	To see \eqref{eq:PhatApprox}, it clearly suffices to show that both
	$$\max_{1\le i \le 2k\e} \left|1-\frac{p_i(\omega)}{p_i^{top}(\omega)}\right|\quad\quad\text{and}\quad\quad \max_{1\le i \le 2k\e} \left|1-\frac{p_i(\omega)}{p_i^{bot}(\omega)}\right|$$
	satisfy the same bound \eqref{eq:etaBound} as $\eta(k,\e,\delta)$ does, uniformly for all $(t,h)$ in $\Delta_{k,\delta}^{top}$ and $\Delta_{k,\delta}^{bot}$, respectively, as well as uniformly for all $\omega\in\Omega$.
	
	Elementary calculations using the assumptions on $(t,h)$, in particular $t\leq 2k(1-\delta-\e)$, and the fact that $|h-h_i^{abs}(\omega)|\leq 2k\e$ as well as $0\leq t_i^{abs}-t\leq 2k\e$ give
	
	$$0\leq \frac{p_i(\omega)}{p_i^{top}(\omega)}-1 = \frac{1}{h_i^{abs}+1}\leq \frac{1}{2k\e(\ell-1)+1}\leq \frac{1}{k\e},\quad\quad\text{for }\ell\geq 2$$
	as well as
	$$0\leq
	1-\frac{p_i(\omega)}{p_i^{bot}(\omega)} =
	\frac{h_i^{abs}}{2k-t_i^{abs}}\leq 
	(\ell+1)\e\frac{1}{1-\e-\frac{t}{2k}}\leq 
	(\ell+1)\frac{\e}{\delta}.$$
	Noting that these bounds do not depend on $t,h,i$ or $\omega$ we have proven the claim.
\end{proof}

Before we prove Lemma \ref{lem:boundOnCondExp} we also show that the time-homogeneous approximation $\P^{bot}_{(t,h)}(\omega)$ can be calculated explicitly in terms of the ``relative height difference" of the path induced by $\omega$. 
\begin{lem}\label{lem:explicitFormulaForTimeHomApproxInA}
	Fix $(t,h)\in\Delta_k$ and $\omega\in\Omega=\{+1,-1\}^{2k\e}$. Let $\Delta\omega\deq \sum_{i=1}^{2k\e}\omega(i)$ be the relative height difference of the path $\pi^h(\omega)$. Then we have
	$$\P_{(t,h)}^{bot}(\omega)\eqbydef
		\prod_{i=1}^{2k\e}
			\left(p_i^{bot}(\omega)\right)^{{\bf 1}(\omega(i)=1)}
			\left(1-p_i^{bot}(\omega)\right)^{{\bf 1}(\omega(i)=-1)}
		= \left(\frac{1}{2}\right)^{2k\e}\frac{h+1+\Delta\omega}{h+1}.$$
\end{lem}
\begin{proof}
Recall that by definition we have $h_0^{abs}=h, h_1^{abs}=h_0^{abs}+\omega(1),\dots$. 
Hence
\begin{align*}
	2^{2k\e}\P_{(t,h)}^{bot}
	&\eqbydef 
		\prod_{i=0}^{2k\e-1} 
			\left(\frac{h_i^{abs}+2}{h_i^{abs}+1}\right)^{{\bf 1}(\omega(i+1)=1)}
			\left(\frac{h_i^{abs}}{h_i^{abs}+1}\right)^{{\bf 1}(\omega(i+1)=-1)}\\
	&=
		\prod_{i=0}^{2k\e-1} 
			\frac{h_i^{abs}+\omega(i+1)+1}{h_i^{abs}+1}\\
	&=
		\frac{h_{2k\e}+1}{h+1},
\end{align*}
where the last step followed by noting that $h_{i+1}^{abs}=h_i^{abs}+\omega(i+1)$ and a telescoping product argument. Since $h_{2k\e}=h+\Delta\omega$ this proves the claim.
\end{proof}

\begin{proof}[Proof of Lemma \ref{lem:boundOnCondExp}]	
	To use Holley's inequality we consider $\Omega = \{+1,-1\}^{2k\e}= \{(,)\}^{2k\e}$, where every $\omega\in\Omega$ is naturally identified with a part of a Dyck path $\pi^h$ via its bracket notation. This induces a partial ordering on the set of subpaths of length $2k\e$ and allows for $\omega^e$ or $\omega_e$ to be interpreted as the increments of the path where the $e$-th step is replaced with an up or down, respectively.
	Then $\omega\leq\omega'$ if and only if $\omega_i\leq\omega'_i$ at every position $i$, i.e. if and only if $\pi(\omega)$ goes down every time $\pi(\omega')$ goes down.
	Furthermore let $\P_\frac{1}{2}$ be the probability measure corresponding to $\E_\frac{1}{2}$, i.e. $$\P_\frac{1}{2}(\omega)\eqbydef 2^{-2k\e}.$$
	Now fix $(t,h)\in\Delta_{k,\delta}$. By Lemma \ref{lem:timeHomApprox} it suffices to show that \be\label{eq:suffCondToCheckWithHolley}\widehat\E_{(t,h)}z^{T_{\ell,\e}(\pi)}\leq \E_\frac{1}{2}z^{U(\pi)}.\ee
	By definition of $T_{\ell,\e}$ and $\widehat\P_{(t,h)}$ we deal with two different cases depending on $h$.
	
	\underline{\textbf{Case: $h\le 2k\e\ell$.}}	
	In this regime we have, by definition of $T_{\ell,\e}$, that $z^{T_{\ell,\e}} = z^U$ and $\widehat\P_{(t,h)}(\omega)\eqbydef \P_{(t,h)}^{bot}(\omega)$.
	Now set\footnote{Note that $U(\pi)$ is independent of the initial height $h=\pi^h(0)=\pi(0)$, hence we will suppress $h$ in the notation.} $$X(\omega)\deq -z^{U(\pi(\omega))},$$
	and we claim that $X$ is increasing. Indeed, this easily follows by induction (for every pair $\omega\leq\omega'$ introduce a sequence $\omega=\omega^{(0)}\le\omega^{(1)}\le\omega^{(2)}\le\dots\le\omega^{(I)}=\omega'$ such that $\omega^{(i)}=\omega^{(i+1)}_{e_i}$ for some $e_i$), using the fact that $z^{U(\pi(\omega^e))}\le z^{U(\pi(\omega_e))}$. This inequality is a consequence of the submultiplicativity of the norm and the definition of $z_j$ in \eqref{eq:def_zj}.
	
	To get \eqref{eq:suffCondToCheckWithHolley} we will apply Holley's inequality to $X$ and $\mu_1=\P_\frac{1}{2},\mu_2=\P_{(t,h)}^{bot}$. Hence it suffices to check that
	\begin{enumerate}
		\item $\P_\frac{1}{2}(\omega^e)\P_{(t,h)}^{bot}(\omega_e)\leq \P_\frac{1}{2}(\omega_e)\P_{(t,h)}^{bot}(\omega^e)$, and
		\item $\mu(\omega_{ef})\mu(\omega^{ef})\geq \mu(\omega^e_f)\mu(\omega_e^f)$ for $\mu=\P_\frac{1}{2}$ or $\mu=\P_{(t,h)}^{bot}$.
	\end{enumerate}
	
	To check condition (1), notice that $\P_\frac{1}{2}(\omega_e)=\P_\frac{1}{2}(\omega^e)$, so it suffices to prove 
	\be\label{eq:ratioOfApproxLeq1}
		\frac{\P_{(t,h)}^{bot}(\omega_e)}{\P_{(t,h)}^{bot}(\omega^e)}\leq 1.
	\ee 
	Without loss of generality we may assume $e=1$ since changing the $e$-th entry in $\omega$ does not change the contribution from the first $e-1$ terms in the product in $$\P_{(t,h)}^{bot}(\omega)\eqbydef\prod_{i=1}^{2k\e}\left(p_i^{bot}(\omega)\right)^{{\bf 1}(\omega(i)=1)}\left(1-p_i^{bot}(\omega)\right)^{{\bf 1}(\omega(i)=-1)}.$$
	Now pick any $\omega=\omega^e$, notice that $\Delta\omega^e = \Delta\omega_e + 2$, and apply Lemma \ref{lem:explicitFormulaForTimeHomApproxInA} to see that \eqref{eq:ratioOfApproxLeq1} holds and hence condition (1) in Holley's inequality is satisfied.
			
	Condition (2) is trivially fulfilled by choosing $\mu=\P_\frac{1}{2}$ for which we have equality.
	Hence Holley's inequality in this setup gives \eqref{eq:suffCondToCheckWithHolley} for $h\leq 2k\e\ell$.
	
	\underline{\textbf{Case: $h>2k\e\ell$.}}	
	Similarly to the previous case we now apply Holley's inequality to $X\deq z^{D(\pi)}$, $\mu_1 = \widehat\P_{(t,h)}(\omega)\eqbydef\P_{(t,h)}^{top}$, and $\mu_2 = \P_\frac{1}{2}$.
	As before, $X$ is increasing 
	and condition (2) of Holley's inequality is trivially fulfilled by $\mu_2 = \P_\frac{1}{2}$. To show condition (1), i.e.
	$$\P_{(t,h)}^{top}(\omega^e)\P_\frac{1}{2}(\omega_e)\leq \P_{(t,h)}^{top}(\omega_e)\P_\frac{1}{2}(\omega^e) 
	\quad\Leftrightarrow\quad
	\P_{(t,h)}^{top}(\omega^e)\leq \P_{(t,h)}^{top}(\omega_e),$$
	we fix any $e\in[1,2k\e)$ and consider for some $\omega=\omega^e$ the ratio
	\be\label{eq:ratioCond1ForHolleyInTopCase}
		\frac{\P^{top}_{(t,h)}(\omega^e)}{\P^{top}_{(t,h)}(\omega_e)} =
		\underbrace{\frac{q_{t_{e-1}^{abs},h_{e-1}^{abs}}}{1-q_{t_{e-1}^{abs},h_{e-1}^{abs}}}}_{\leq 1}\prod_{i=e}^{2k\e-1}
			{\underbrace{\left(\frac{q_{t_i^{abs},h_i^{abs}}}{q_{t_i^{abs},h_i^{abs}-2}}\right)}_{\leq 1}}^{{\bf 1}(\omega(i)=1)}
			{\underbrace{\left(\frac{1-q_{t_i^{abs},h_i^{abs}}}{1-q_{t_i^{abs},h_i^{abs}-2}}\right)}_{\geq 1}}^{{\bf 1}(\omega(i)=-1)},
	\ee
	where 
	\be\label{eq:defOfqs}
		q_{t,h}\deq \frac{1}{2}\frac{2k-t-h}{2k-t}
		\quad\quad
		1-q_{t,h} = \frac{1}{2}\frac{2k-t+h}{2k-t}.
	\ee	
	Note that $h_i^{abs}\eqbydef h_i^{abs}(\omega)=h_i^{abs}(\omega^e)$ since we assumed $\omega=\omega^e$ 
	and $h_i^{abs}(\omega^e)=h_i^{abs}(\omega_e)+2\cdot{\bf 1}(i \geq e)$.
	
	Now we show that \eqref{eq:ratioCond1ForHolleyInTopCase} is less or equal than $1$ for all choices of $e$ and $\omega$. Since the first factor in the product in \eqref{eq:ratioCond1ForHolleyInTopCase} is less or equal than $1$ we can bound
	\begin{align*}
		\frac{\P^{top}_{(t,h)}(\omega^e)}{\P^{top}_{(t,h)}(\omega_e)} 
		&\leq \frac{q_{t_{e-1}^{abs},h_{e-1}^{abs}}}{1-q_{t_{e-1}^{abs},h_{e-1}^{abs}}}
			\prod_{i=e}^{2k\e-1}
			\left(\frac{1-q_{t_i^{abs},h_i^{abs}(\omega)}}{1-q_{t_i^{abs},h_i^{abs}(\omega)-2}}\right)^{{\bf 1}(\omega(i)=-1)}\\
		&\leq \Bigg(\max_{j\in[0,2k\e),\atop \tilde\omega\in\Omega}\frac{q_{t_{j}^{abs},h_{j}^{abs}(\tilde\omega)}}{1-q_{t_{j}^{abs},h_{j}^{abs}(\tilde\omega)}}\Bigg)
			\prod_{i=1}^{2k\e-1}
			\frac{1-q_{t_i^{abs},h_i^{abs}(\omega')}}{1-q_{t_i^{abs},h_i^{abs}(\omega')-2}}.\numberthis\label{eq:maximisethis}
	\end{align*}
	where $\omega'\deq(1,-1,-1,\dots,-1)$. The second inequality holds because setting $e=1$ in the product gives more factors that are greater or equal to one and because all factors for $i\geq e$ can be directly compared (for different $\omega$) using
	$$1\leq \frac{1-q_{t,h+c}}{1-q_{t,h-2+c}}\leq \frac{1-q_{t,h}}{1-q_{t,h-2}},\quad\quad \text{for all } c\geq 0, \;(t,h)\in\Delta_k.$$
	Since $q\mapsto\frac{q}{1-q}$ is monotonically increasing for $0\leq q<1$, and $(t,h)\mapsto q_{t,h}$ is monotonically decreasing in both variables, we have
	$$\max_{j\in[0,2k\e),\atop \tilde\omega\in\Omega}\frac{q_{t_{j}^{abs},h_{j}^{abs}(\tilde\omega)}}{1-q_{t_{j}^{abs},h_{j}^{abs}(\tilde\omega)}}
	\leq \max_{\Delta t\in[0,m],\atop\Delta h\in[-m,m]}\frac{q_{t+\Delta t,h+\Delta h}}{1-q_{t+\Delta t,h+\Delta h}}
	= \frac{q_{t,h-m}}{1-q_{t,h-m}}
	= \frac{2k-t-h+m}{2k-t+h-m}
	$$
	for $m\deq 2k\e-1$.
	To show that \eqref{eq:maximisethis} is less or equal than $1$, it suffices (using the formulas \eqref{eq:defOfqs} and $t_i^{abs}\eqbydef t+i,h_i^{abs}(\omega')\eqbydef h-i+2$) to see that 
	$$
		\frac{2k-t-h+m}{2k-t+h-m}\prod_{i=1}^m \frac{2k-t+h+2-2i}{2k-t+h-2i} \leq 1.
	$$
	This is easy to see by a telescoping product argument and using that $h\geq 2m$ (since $4k\e\leq 2k\e\ell<h$ by assumption).
	Thus $\P_{(t,h)}^{top}(\omega^e)\leq\P_{(t,h)}^{top}(\omega_e)$ so that condition (1) of Holley's inequality is satisfied.
	Hence we can apply Holley's inequality to get \eqref{eq:suffCondToCheckWithHolley} for $h>2k\e\ell$.
	This completes the proof of the lemma.
\end{proof}
\subsection{Calculating the $\limsup$}
Now we are ready to prove the main result:

\begin{proof}[Proof of Theorem \ref{thm:main}]
Start with \eqref{with2k} and note that the expectation on the r.h.s. can be written as in \eqref{eq:nestingOfExpectations}. We apply the trivial bound $z^{T_{\ell,\e}(\pi)}\leq 1$ for the first $\delta/\e$ innermost terms in \eqref{eq:nestingOfExpectations} and Lemma \ref{lem:boundOnCondExp} $\frac{1}{\e}(1-\delta)$ times to the remaining ones, yielding

\begin{align*}
	&\big[ \E \; val(\Gamma)\big]^{1/2k} \\&\le 
	\frac{L^{1/2k}}{  (\min_{j\le J} z_j^{1/j})^{1/L} }  \| S\|^{1/2}\cdot
   \max_{(t,h)\in\Delta_{k,\delta}\atop \ell\in[M,M+2L]}
   \left[ \prod_{j=0}^{\frac{1-\delta}{\e}-1}\left((1+\eta_{t,h,\ell}(k,\e,\delta))^{2k\e}\E_\frac{1}{2}z^{U(\pi_j)}\right) \right]^{1/2k}\\
   &=\frac{L^{1/2k}}{  (\min_{j\le J} z_j^{1/j})^{1/L} }  \| S\|^{1/2}\cdot
   \left(1+(M+2L+1)\frac{\e}{\delta}+\frac{1}{k\e}\right)^{1-\delta}\big[ \E_\frac{1}{2}z^{U(\pi_0)} \big]^{\frac{1-\delta}{2k\e}},
\end{align*}
where $\pi_0$, the random variable over which we are taking expectation, is a path of length $2k\e$. There we used that since $\ell = M+2m$ is between $M$ and $M+2L$, we have $\eta_{t,h,\ell}(k,\e,\delta)\leq (M+2L+1)\frac{\e}{\delta}+\frac{1}{k\e}$ from \eqref{eq:etaBound}, uniformly in $t,h$.

After taking limits in the following order
$$
	\lim_{L,M\to\infty}\lim_{\e,\delta\to 0: \atop \e/\delta\to 0} \limsup_{k\to\infty},
$$
using that $\lim_k |\mathcal T_k|^{1/2k}=2$ and the change of variables $n=2k\e$ 
we have for $\pi^{(n)}$ a simple random walk of length $n$, as in \eqref{eq:suppRhoNbddByLimsupValN}:
\be\label{eq:limsupbnd}
\max\supp\rho 
= \limsup_k\big[ |\mathcal T_k|\cdot\E \; val(\Gamma)\big]^\frac{1}{2k}
\leq 2\|S\|^\frac{1}{2}\limsup_{n\to\infty}\left(\E_\frac{1}{2}z^{U(\pi^{(n)})}\right)^\frac{1}{n}.
\ee
To estimate $(\E_\frac{1}{2}z^{U(\pi^{(n)})})^\frac{1}{n}$ for large $n$, we introduce a randomised stopping time $n^*$ with geometric distribution $\P(n^*=m)=w^{m-1}(1-w)$, where $w$ is a new parameter to be optimised later. Denote the expectation over $n^*$ by $\E^*$. We set $U^* \deq U(\pi^{(n^*-1)})$.
Following Theorem 2 in \cite{HK14} we find that for $J\in\N$ we have
\be\label{eq:explosiveCoinIdentity}
\E^*\E_\frac{1}{2}z^{U^*} = (1-w)\frac{1+\sum_{j=1}^J \left(\frac{w}{2}\right)^j z_j + \sum_{j>J}\left(\frac{w}{2}\right)^j}{1-\frac{w}{2}(1+\sum_{j=1}^J\left(\frac{w}{2}\right)^jz_j+\sum_{j>J}\left(\frac{w}{2}\right)^j)},
\ee
as well as $$\E^*\E_\frac{1}{2}z^{U^*} = (1-w)\sum_{n\ge 0}w^n\E_\frac{1}{2} z^{U(\pi^{(n)})},$$
as in equation (10) in \cite{HK14}. Interpreting $\E_\frac{1}{2} z^{U(\pi^{(n)})}$ as the coefficients of the power series (in $w$) of $\E^*\E_\frac{1}{2}z^{U^*}$, it suffices to find (the inverse of) its radius of convergence to get the $\limsup$ in \eqref{eq:limsupbnd} by Cauchy-Hadamard.

Considering the explicit formula \eqref{eq:explosiveCoinIdentity}, note that this radius of convergence is equal to the minimum of $2$ and $w_c$ being defined as the smallest (in absolute value) root of the denominator in \eqref{eq:explosiveCoinIdentity}, i.e. the function $\phi_J$ defined in \eqref{delphi}.
It is easy to see that the smallest (in absolute value) root of $\phi_J$ is positive and smaller than $2$.
Thus $\limsup_{n\to\infty}\left(\E_\frac{1}{2}z^{U(\pi^{(n)})}\right)^\frac{1}{n} = \frac{1}{w_c}$, proving Theorem \ref{thm:main}.
\end{proof}

\appendix
\section{Numerics}
For $N=500$, $J=50$, and $S_{ij}\deq e^{\frac{i+j}{N}}$ the trivial bound $2\|S\|^{1/2}\approx 4.316$ and the empirical average (number of samples = 10) of the largest eigenvalue (in absolute value) is $\approx 3.677$ (with empirical standard deviation of $\approx 0.047$). Our method improves the trivial bound to $\approx 3.870$, a factor of improvement of $w_c \approx 1.115$.

\end{document}